\pdfoutput=1
\RequirePackage{ifpdf}
\ifpdf % We~are running pdfTeX in pdf mode
\documentclass[pdftex]{sigma}%, draft
\else
\documentclass{sigma}
\fi

\numberwithin{equation}{section}

\newtheorem{Theorem}{Theorem}[section]
\newtheorem{Proposition}[Theorem]{Proposition}

\usepackage{mathrsfs}

\newcommand{\pa}{\partial}

\begin{document}
\allowdisplaybreaks

\newcommand{\arXivNumber}{2104.00895}

\renewcommand{\thefootnote}{}

\renewcommand{\PaperNumber}{083}

\FirstPageHeading

\ShortArticleName{Resolvent Trace Formula and Determinants of $n$ Laplacians}

\ArticleName{Resolvent Trace Formula and Determinants\\ of $\boldsymbol{n}$ Laplacians on Orbifold Riemann Surfaces\footnote{This paper is a~contribution to the Special Issue on Mathematics of Integrable Systems: Classical and Quantum in honor of Leon Takhtajan.

The full collection is available at \href{https://www.emis.de/journals/SIGMA/Takhtajan.html}{https://www.emis.de/journals/SIGMA/Takhtajan.html}}}

\Author{Lee-Peng TEO}

\AuthorNameForHeading{L.-P.~Teo}

\Address{Department of Mathematics, Xiamen University Malaysia,\\ Jalan Sunsuria, Bandar Sunsuria, 43900, Sepang, Selangor, Malaysia}
\Email{\href{mailto:lpteo@xmu.edu.my}{lpteo@xmu.edu.my}}

\ArticleDates{Received April 07, 2021, in final form September 05, 2021; Published online September 13, 2021}

\Abstract{For $n$ a nonnegative integer, we consider the $n$-Laplacian $\Delta_n$ acting on the space of $n$-differentials on a confinite Riemann surface $X$ which has ramification points. The trace formula for the resolvent kernel is developed along the line \`a la Selberg. Using the trace formula, we compute the regularized determinant of $\Delta_n+s(s+2n-1)$, from which we deduce the regularized determinant of $\Delta_n$, denoted by $\det\!'\Delta_n$. Taking into account the contribution from the absolutely continuous spectrum, $\det\!'\Delta_n$ is equal to a constant $\mathcal{C}_n$ times $Z(n)$ when $n\geq 2$. Here $Z(s)$ is the Selberg zeta function of $X$. When $n=0$ or $n=1$, $Z(n)$ is replaced by the leading coefficient of the Taylor expansion of $Z(s)$ around $s=0$ and~$s=1$ respectively. The constants $\mathcal{C}_n$ are calculated explicitly. They depend on the genus, the number of cusps, as well as the ramification indices, but is independent of the moduli parameters.}


\Classification{14H15; 11F72; 11M36}

\begin{flushright}
\begin{minipage}{65mm}
\it Dedicated to Professor Leon Takhtajan\\ on the occasion of his 70th birthday
\end{minipage}
\end{flushright}

\renewcommand{\thefootnote}{\arabic{footnote}}
\setcounter{footnote}{0}

\section{Introduction}

Let $\mathbb{H}$ be the upper half plane, and let $\Gamma$ be a Fuchsian group, a discrete subgroup of ${\rm PSL}(2,\mathbb{R})$. The group $\Gamma$ acts discontinuously on $\mathbb{H}$ and the quotient $X=\Gamma\backslash\mathbb{H}$ is a Riemann surface. In this work, we consider the case where $X$ is a cofinite Riemann surface. In other words, $X$ has finite hyperbolic volume. This includes compact Riemann surfaces, as well as surfaces with finitely many cusps and ramification points.

Let
\[
\Delta_0=-y^2\bigg(\frac{\pa^2}{\pa x^2}+\frac{\pa^2}{\pa y^2}\bigg)
\] be the Laplacian operator acting on functions. It is an invariant operator under the action of ${\rm PSL}(2,\mathbb{R})$.

\looseness=-1 In the seminal paper \cite{Selberg1956}, Selberg developed a theory that can be used to study the spectrum of $\Delta_0$ on Riemann surfaces, making full use of the invariant property of the Laplacian operator. The theory was subsequently developed and elaborated in \cite{Fischer1987, Hejhal1976, Hejhal1983,Selberg, Venkov1982,VenkovKalininFaddeev,Voros1987}.
One of the extensions of Selberg's theory is to consider the spectrum of $D_n$ for $n\geq 1$, where $D_n=-K_{n-1}L_n-n(n-1)$, and $K_n$ and $L_n$ are the Maass operators\footnote{Our $D_n$ has a negative sign compared to that used in \cite{Fay}, while Fischer \cite{Fischer1987} denote it by $-\Delta_n$.} (see Section~\ref{sec2} for more details).

The significance of this theory is an identity that relates the spectral trace to the geometric trace.

In mathematical physics, the regularized determinant of the Laplacian is of special interest. This has been considered in \cite{DHokerPhong1986, Sarnak1987} for compact Riemann surfaces. Subsequently, Efrat \cite{Efrat1988} exten\-ded the result to smooth Riemann surfaces with cusps. In the special case of arithmetic surfaces, this has also been considered in \cite{Koyama1991, Koyama1991_2}. In \cite{Gong1995}, Gong derived the regularized determinant in~full generality
using the trace formula developed in \cite{Fischer1987} directly, which can be applied to cofinite Riemann surfaces and for the operators $D_n$ with unitary twists.

The goal of this paper is to study the explicit expression for the resolvent trace formula and the regularized determinant of $n$-Laplacians $\Delta_n$ on Riemann surfaces that have finite volumes. $\Delta_n$ has the same spectrum as $D_n+n(n-1)$. In principle, one can extract the expression for the trace formula and the determinant of $\Delta_n$ from the formulas obtained in \cite{Fischer1987, Gong1995, Hejhal1983, Venkov1982}. However, we find that a direct approach from a theory for $\Delta_n$ would be more appealing for further application. This is the tasks undertaken in this paper.

\section[Laplacians of n-differential]{Laplacians of $\boldsymbol{n}$-differentials}\label{sec2}

Let $\mathbb{H}$ be the upper half of the complex plane equipped with the hyperbolic metric
\begin{gather*}
{\rm d}s^2=\frac{{\rm d}x^2+{\rm d}y^2}{y^2}.
\end{gather*}
The corresponding metric density is $\rho(z)=(\operatorname{Im}z)^{-2}$ and the area form is
\begin{gather*}
{\rm d}\mu(z)=\frac{{\rm d}x\,{\rm d}y}{y^2}=\rho(z)\,{\rm d}^2z.
\end{gather*}
The group ${\rm PSL}(2,\mathbb{R})$ acts on $\mathbb{H}$ transitively by M$\ddot{\text{o}}$bius transformations
\begin{gather*}
\gamma=\begin{pmatrix}a & b\\c & d\end{pmatrix}\!\colon \ z\mapsto \gamma(z)=\frac{az+b}{cz+d}.
\end{gather*}
The hyperbolic metric is an invariant metric under this group action.

The distance between two points $z$ and $w$ on $\mathbb{H}$, denoted by $d(z,w)$, is given by
\begin{gather*}
\cosh d(z,w)=1+2u(z,w),
\end{gather*}
where
\begin{gather*}
u(z,w)=\frac{|z-w|^2}{4\operatorname{Im} z\operatorname{Im}w}
\end{gather*}
is a point-pair invariant, i.e.,
\begin{gather*}
u(\gamma z, \gamma w)=u(z,w)\qquad \text{for all}\quad \gamma\in{\rm PSL}(2,\mathbb{R}).
\end{gather*}

Let $\Gamma$ be a cofinite Fuchsian group, namely, it is a discrete subgroup of ${\rm PSL} (2, \mathbb{R})$ such that $X=\Gamma\backslash\mathbb{H}$ is a Riemann surface with finite hyperbolic volume. Such $\Gamma$ is finitely generated. More precisely, if $X$ is a genus $g$ Riemann surface with $q$-punctures and $v$ ramification points, then~$\Gamma$ is generated by $2g$ hyperbolic elements $\alpha_1, \beta_1, \dots, \alpha_g$, $\beta_g$, $q$ parabolic elements $\kappa_1, \dots, \kappa_q$, as well as $v$ elliptic elements $\tau_1, \dots, \tau_v$ of orders $m_1, \dots, m_v$ respectively, with
\begin{gather*}
2\leq m_1\leq m_2\leq \dots \leq m_v.
\end{gather*}
For each $1\leq j\leq v$, $\tau_j^{m_j}=I$. The generators of $\Gamma$ satisfy the additional relation
\begin{gather*}
\alpha_1\beta_1\alpha_1^{-1}\beta_1^{-1}\cdots \alpha_g\beta_g\alpha_g^{-1}\beta_g^{-1}\kappa_1\cdots\kappa_q\tau_1\cdots\tau_v=I,
\end{gather*}
where $I$ is the identity element. We say that the Riemann surface $X$ and the group $\Gamma$ are of type $(g;q;m_1, m_2, \dots, m_v)$.

Let $\mathcal{K}$ be the canonical bundle of $X$.
For a nonnegative integer $n$, let $S(n)$ be the space of sections of $\mathcal{K}^{n/2}\otimes \bar{\mathcal{K}}^{-n/2}$. A function in $S(n)$ can be realized as a function $f\colon \mathbb{H}\rightarrow\mathbb{C}$ satisfying
\begin{gather*}
f(\gamma z)\bigg(\frac{c\bar{z}+ d}{cz+d}\bigg)^n =f(z),\qquad \text{for all}\quad \gamma=\begin{pmatrix}a & b\\c & d\end{pmatrix}\in\Gamma.
\end{gather*}
In the context of analytic number theory, it is natural to consider the Maass operators $K_n$: $S(n)\rightarrow S(n+1)$, $L_n\colon S(n)\rightarrow S(n-1)$ and $D_n\colon S(n)\rightarrow S(n)$ given by~\cite{Fay}
\begin{gather*}
K_n=\left(z-\bar{z}\right)\frac{\pa}{\pa z}+n,
\\
L_n=-\left(z-\bar{z}\right)\frac{\pa}{\pa \bar{z}}-n,
\\
D_n=-y^2\bigg(\frac{\pa^2}{\pa x^2}+\frac{\pa^2}{\pa y^2}\bigg)+2{\rm i}ny\frac{\pa}{\pa x}.
\end{gather*}
These operators are invariant with respect to the ${\rm PSL}(2,\mathbb{R})$ action and
\begin{gather*}
D_{n+1}K_n=K_{n}D_n,\qquad D_{n-1}L_n=L_{n}D_{n},
\\
D_n=-L_{n+1}K_n-n(n+1)=-K_{n-1}L_n-n(n-1).
\end{gather*}
For applications of the theory of Riemann surfaces in physics, it is more natural to consider the space of tensors $f(z)({\rm d}z)^n$, or called $n$-differentials, on $X$. These are sections of $\mathcal{K}^n$, which are functions $f\colon \mathbb{H}\rightarrow\mathbb{C}$ satisfying
\begin{gather*}
f(\gamma z)\gamma'(z)^n=f(z),\qquad \text{for all}\quad \gamma\in\Gamma.
\end{gather*}
We denote this space by $H_n^2(\Gamma)$. It is a Hilbert space with the inner product
\begin{gather*}
\langle f, g\rangle =\iint_X f(z)\overline{g(z)}\rho(z)^{-n}{\rm d}\mu(z).
\end{gather*}
The $n$-Laplacian operator $\Delta_n=4\bar{\partial}^*_n\bar{\partial}_n$ is an operator on $H_n^2(\Gamma)$ with explicit formula given by
\begin{gather*}
\Delta_n= -4y^{2-2n}\frac{\pa}{\pa z}y^{2n}\frac{\pa}{\pa\bar{z}}.
\end{gather*}
This is a positive operator. We put a factor $4$ in front of $\bar{\partial}^*_n\bar{\partial}_n$ so that when $n=0$, we get the usual Laplacian on functions $\Delta_0$.

Throughout this paper, the nonnegative integer $n$ is fixed. Our goal is to study the trace formula for the resolvent of $\Delta_n$ and the regularized determinant of $\Delta_n$.
The isometry $I$: \mbox{$H_n^2(\Gamma)\rightarrow S(n)$} defined by
\begin{gather*}
f(z) \mapsto y^n f(z)
\end{gather*}
conjugates $\Delta_n$ with $D_n+n(n-1)$. Hence, $\Delta_n$ and $D_n+n(n-1)$ have the same spectrum.

For a Riemann surface with cusps, the spectrum of $\Delta_n$ consists of a discrete part and a~continuous part. To study the spectrum of the continuous part, we need to consider Eisenstein series, which we discuss in next section.

A powerful tool to study the spectrum of $\Delta_n$ is the Selberg trace formula.
The Selberg trace formula for the operator $D_n$ has been developed extensively in \cite{Fischer1987, Hejhal1983, Venkov1982}. It can be adapted to $\Delta_n$. To find the determinants of Laplacian, we will follow Fischer \cite{Fischer1987} and derive the trace formula for the resolvent kernel of the Laplacian operator first.

Let $\Psi\colon \mathbb{R}\rightarrow \mathbb{C}$ be a function that vanishes at infinity and satisfies other regularity conditions to be specified when needed. Define $k\colon \mathbb{H}\times\mathbb{H}\rightarrow\mathbb{C}$ by
\begin{equation}\label{eq0903_2}
k(z,w)=\Psi(u(z,w))\frac{(-4)^n}{(z-\bar{w})^{2n}}=\Psi(u(z,w))H_n(z,w),\end{equation}
where
\begin{gather*}H_n(z,w)=\frac{(-4)^n}{(z-\bar{w})^{2n}}.\end{gather*}
Notice that $H_n(z,z)=\rho(z)^n$.
Since $H_n(\gamma z, \gamma w)\gamma'(z)^n\overline{\gamma'(w)}^n=H_n(z,w)$, we find that
\begin{gather*}
k(\gamma z, \gamma w)\gamma'(z)^n\overline{\gamma'(w)}^n=k(z,w).
\end{gather*}
In other words, $k(z,w)$ is a point-pair invariant kernel on~$\mathbb{H}$.

Given $s\in \mathbb{C}$, the function $f(z)=(\operatorname{Im}(z))^{s-n}$ is an ``eigenfunction'' of $\Delta_n$ with eigenvalue $-(s-n)(s+n-1)$. The following proposition shows that $f(z)$ is also an eigenfunction for the operator defined by the point-pair invariant kernel $k(z,w)$.

\begin{Proposition} \label{P1}
Let $\Psi\colon \mathbb{R}\rightarrow \mathbb{C}$ be a continuous function with compact support and let \begin{gather*}
k(z,w)= \Psi(u(z,w))H_n(z,w)
\end{gather*}
be the corresponding point-pair invariant kernel. If $s\in \mathbb{C}$, then
\begin{gather*}
\iint_{\mathbb{H}}k(z,w)(\operatorname{Im} w)^{s-n}(\operatorname{Im}w)^{2n}{\rm d}\mu(w)=\Lambda_s (\operatorname{Im} z)^{s-n},
\end{gather*}
where
\begin{gather*}
\Lambda_s =\iint_{\mathbb{H}}k({\rm i},w)(\operatorname{Im} w)^{s-n}(\operatorname{Im}w)^{2n}{\rm d}\mu(w).
\end{gather*}
\end{Proposition}
\begin{proof}
Given $z\in\mathbb{H}$, notice that $\gamma=\big(\begin{smallmatrix} a & b \\ 0 & d\end{smallmatrix}\big)$, with $ \frac{1}{d}=a=\sqrt{\operatorname{Im}z}$ and $ab=\operatorname{Re}z$ maps ${\rm i}$ to $z$. For such $\gamma$, $\gamma'(z)=a^2=\operatorname{Im}z$.
Let
\begin{gather*}g(z)=\iint_{\mathbb{H}}k(z,w)(\operatorname{Im} w)^{s-n}(\operatorname{Im}w)^{2n}{\rm d}\mu(w).\end{gather*}Using $k(\gamma {\rm i},\gamma w)\gamma'({\rm i})^n\overline{\gamma'(w)}^n=k({\rm i},w)$ and $\operatorname{Im}(\gamma w)=\operatorname{Im}(w)|\gamma'(w)|$, a change of variables $w\mapsto \gamma w$ gives
\begin{align*}
g(z)
 &= \iint_{\mathbb{H}}k(\gamma {\rm i},\gamma w)(\operatorname{Im} \gamma w)^{s+n}{\rm d}\mu(w)
 = a^{2s-2n}\iint_{\mathbb{H}}k({\rm i},w)(\operatorname{Im} w)^{s+n}{\rm d}\mu(w)
 \\
 &= \Lambda_s (\operatorname{Im}z )^{s-n},
\end{align*}
where
\begin{gather*}
\Lambda_s=\iint_{\mathbb{H}}k({\rm i},w)(\operatorname{Im} w)^{s+n}{\rm d}\mu(w).\tag*{\qed}
\end{gather*}\renewcommand{\qed}{}
\end{proof}

Given the point-pair invariant kernel $k(z,w)$ on $\mathbb{H}$, we can define a point-pair invariant ker\-nel~$K(z,w)$ for the Riemann surface $X$ by
\begin{gather}\label{eq0903_3}
K(z,w)=\sum_{\gamma\in\Gamma}k(\gamma z, w)\gamma'(z)^n, \qquad z, w\in\mathbb{H}.
\end{gather}
Notice that if $f\in H_n^2(\Gamma)$, then
\begin{gather*}\iint_X K(z,w)f(w)\rho(w)^{-n}{\rm d}\mu(w)=\iint_{\mathbb{H}}k(z,w)f(w)\rho(w)^{-n}{\rm d}\mu(w).\end{gather*}

 We want to find a function $\Psi_{n,s}(u)=\Psi(u)$ so that the corresponding $K(z,w)$ is the resolvent kernel of the operator
 $\Delta_n+s(s+2n-1)$. The reason to use $s(s+2n-1)$ instead of $(s-n)(s+n-1)$ is so that when $s=0$, we get $\Delta_n$. This function $\Psi(u)$ must satisfy the differential equation
\begin{gather*}
\Delta_n(\Psi(u(z,w))H_n(z,w))=-s(s+2n-1)\Psi(u(z,w))H_n(z,w).
\end{gather*}Using the fact that
\begin{gather*}
\Delta_n(\Psi(u(z,w))H_n(z,w))
=\bigg({-}u(u+1)\frac{\pa^2 \Psi}{\pa u^2}+[(2n-2)u-1]\frac{\pa \Psi}{\pa u} \bigg)H_n(z,w),
\end{gather*}
we find that
\begin{gather*}
u(u+1)\frac{\pa^2 \Psi}{\pa u^2}-[(2n-2)u-1]\frac{\pa \Psi}{\pa u} -s(s+2n-1)\Psi=0.
\end{gather*}
A solution is given by
\begin{gather}\label{eq0903_1}
\Psi_{n,s}(u)= \frac{(u+1)^{-s}}{4\pi}\frac{\Gamma(s)\Gamma(s+2n)}{\Gamma(2s+2n)}\,{}_2F_1\left(\begin{matrix} s,s+2n\\2s+2n\end{matrix}\,;\,\frac{1}{u+1}\right)\!,
\end{gather}
where
\begin{gather*}
{}_2F_1\left(\begin{matrix} a,b\\c\end{matrix}\,;\,z\right)=\frac{\Gamma(c)}{\Gamma(a)\Gamma(b)}
\sum_{k=0}^{\infty}\frac{\Gamma(a+k)\Gamma(b+k)}{\Gamma(c+k)}\frac{z^k}{k!}
\end{gather*}
is the hypergeometric function. The normalization constant $1/4\pi$ is chosen so that when $u\rightarrow 0^+$,
\begin{gather*}
\Psi(u)\sim \frac{1}{4\pi} \log\frac{1}{u}.
\end{gather*}
More explicitly, one can show that as $u\rightarrow 0^+$,
\begin{gather}\label{eq0903_13}
\Psi_{n,s}(u)= \frac{1}{4\pi}\bigg\{\log\frac{1}{u}+2\psi(1)-\psi(s+2n)-\psi(s)\bigg\}+O(u).
\end{gather}Here
\begin{gather*}\psi(s)=\frac{\Gamma'(s)}{\Gamma(s)}\end{gather*}
 is the logarithmic derivative of the gamma function $\Gamma(s)$, and $\psi(1)=-\gamma$, where $\gamma$ is the Euler constant.

\section{The Eisenstein series}
For Riemann surfaces that are not compact, it is well-known that the spectrum of the Laplacians contain a continuous part, which are related to the Eisenstein series.

 The Riemann surface $X$ has $q$ cusps corresponding to the $q$ parabolic elements $\kappa_1, \dots, \kappa_q$. For $1\leq i\leq q$, let $x_i\in \mathbb{R}\cup\{\infty\}$ be the fixed point of $\kappa_i$. Then $x_i$ is a representative of the cusp associated to $\kappa_i$. Let $\sigma_i\in{\rm PSL}(2,\mathbb{R})$ be an element that conjugates $\kappa_i$ to
 $ \big(\begin{smallmatrix} 1 & \pm1\\ 0 & 1\end{smallmatrix}\big)$, namely,
\begin{gather*}
\sigma_i^{-1}\kappa_i\sigma_i=\begin{pmatrix} 1 & \pm1\\ 0 & 1\end{pmatrix}\!.
\end{gather*}
Then $\sigma_i(\infty)=x_i$. If $B$ is the parabolic subgroup generated by $\big(\begin{smallmatrix} 1 & 1\\ 0 & 1\end{smallmatrix}\big)$, then $\Gamma_i=\sigma_i B\sigma_i^{-1}$ is the stabilizer of the cusp $x_i$ in $\Gamma$.

Define the Eisenstein series associated to the cusp $x_i$ by
\begin{gather*}
E_i(z,s; n)=\sum_{\gamma\in \Gamma_i\backslash\Gamma}\big[\operatorname{Im}\big(\sigma_i^{-1}\gamma z\big)\big]^{s-n}\big[\big(\sigma_i^{-1} \gamma\big)'(z)\big]^{n},
\end{gather*}
when $\operatorname{Re}s>1$. Here to simplify notation, we write $\sigma^{-1}\circ\gamma$ as~$\sigma^{-1}\gamma$, and we write $\big(\sigma^{-1}\circ\gamma\big)(z)$ as~$\sigma^{-1}\gamma z$. One can check that the definition of the Eisenstein series here differs from the one used in \cite{Fischer1987, Hejhal1983, Venkov1982} by the factor $y^{n}$. This makes good sense in view of the isometry between~$ H_n^2(\Gamma)$ and~$S(n)$ that we discussed earlier.

Using the fact that $\Delta_n y^{s-n}=-(s-n)(s+n-1)y^{s-n}$, and $\Delta_n$ is invariant with respect to the action of ${\rm PSL}(2,\mathbb{R})$, we find that
\begin{gather*}\Delta_n E_i(z,s;n)=-(s-n)(s+n-1)E_i(z,s;n).\end{gather*}
The theory of Eisenstein series has been developed extensively in the books \cite{Fischer1987, Hejhal1983, Iwaniec2002}. In the following, we follow \cite{Iwaniec2002} to give a brief exposition of the facts needed in this work.

Given $1\leq i, j\leq q$, there is a double coset decomposition of the group $\sigma_i^{-1}\Gamma\sigma_j$ into disjoint double cosets given by
\begin{gather*} \sigma_i^{-1}\Gamma\sigma_j=\delta_{ij}\Omega_{\infty}\cup\bigcup_{c>0}\bigcup_{d\,\text{mod}\, c}\Omega_{d/c}.
\end{gather*}
Here
$\Omega_{\infty}$ contains all the upper triangular matrices in $ \sigma_i^{-1}\Gamma\sigma_j$, and $\Omega_{d/c}$ is the double coset $B\omega_{d/c}B$, where $\omega_{d/c}$ is an element of $ \sigma_i^{-1}\Gamma\sigma_j$ of the form
\begin{gather*}
\omega_{d/c}=\begin{pmatrix} * & *\\c& d\end{pmatrix}\!.
\end{gather*}
Using this, one can show that when $y\rightarrow\infty$,
\begin{gather*}
E_i(\sigma_j z,s;n)\sigma_j'(z)^n=\delta_{ij}y^{s-n}+\varphi_{ij}(s;n)y^{1-s-n}+\text{exponentially decaying terms},
\end{gather*}
where
\begin{gather*}
\varphi_{ij}(s;n)=\sqrt{\pi}\frac{\Gamma(s)\Gamma\big(s-\frac{1}{2}\big)} {\Gamma(s+n)\Gamma(s-n)}\sum_{c>0}\sum_{d\,\text{mod}\,c}\frac{1}{c^{2s}}
\end{gather*}
is a Dirichlet series.
From this explicit expression, one finds that
{\samepage\begin{gather}\label{eqphi} \varphi_{ij}(s;n)=\frac{\Gamma(s)^2}{\Gamma(s+n)\Gamma(s-n)}\varphi_{ij}(s;0).
\end{gather}

}\noindent
Namely, there is a simple relation between $\varphi_{ij}(s;n)$ and $\varphi_{ij}(s;0)$.

The $q\times q$ matrix $\Phi(s)=[\varphi_{ij}(s)]$ is called the scattering matrix and it plays an important role in the spectral theory. It has the following properties:
\begin{gather*}
\overline{\Phi(s)}=\Phi(\bar{s}),\qquad \Phi(s)^T=\Phi(s).
\end{gather*}
Moreover, if we denote by $\Xi(z,s;n)$ the column matrix with components $E_i(z,s;n)$, then
\begin{gather*}
\Xi(z,s;n)=\Phi(s;n)\,\Xi(z,1-s;n).
\end{gather*}
It follows that
\begin{gather*}
\Phi(s)\Phi(1-s)=I_q.
\end{gather*}
Denote by $\varphi(s)$ the determinant of $\Phi(s)$, namely, $\varphi(s)=\det\Phi(s).$ Then we find that
\begin{gather*}
\overline{\varphi(s)}=\varphi(\bar{s}),\qquad
\varphi(s)\varphi(1-s)=1.
\end{gather*}

\section{The trace of the resolvent kernel}
The spectral theory for the Riemann surface $X=\Gamma\backslash\mathbb{H}$ states that
 there is a countable orthonormal system $\{u_k\}_{k\geq 0}$ of eigenfunctions of $\Delta_n$ with eigenvalues $0=\lambda_0\leq \lambda_1\leq \lambda_2\leq \cdots$, and eigenpackets given by the Eisenstein series so that for any $f\in H_n^2(\Gamma)$,
\begin{align*}
f(z)= \sum_{k=0}^{\infty}\langle f, u_k\rangle u_k(z) +\frac{1}{4\pi}\sum_{j=1}^q\int_{-\infty}^{\infty}\bigg\langle f, E_j\bigg(\cdot,\frac{1}{2}+{\rm i}r;n\bigg)\bigg\rangle E_j\bigg(z,\frac{1}{2}+{\rm i}r;n\bigg){\rm d}r.
\end{align*}
We want to use this formula to find the trace of the resolvent $(\Delta_n+s(s+2n-1))^{-1}$. However, the kernel function $K(z,w)$ for this resolvent defined by $\Psi_{n,s}(u)$ in \eqref{eq0903_1} has singularity along $z=w$. To circumvent this problem, we consider the kernel function $K(z,w)$ with $\Psi(u)=\Psi_{n,s}(u)-\Psi_{n,a}(u)$, for some fixed $a$. For this kernel function,
\begin{gather*}
\langle K(\cdot, w), u_k\rangle = \Lambda(\lambda_k)\overline{u_k(w)},
\\
\bigg\langle K(\cdot, w), E_j\bigg(\cdot,\frac{1}{2}+{\rm i}r;n\bigg)\bigg\rangle= \widetilde{\Lambda}(r)\overline{E_j\bigg(w,\frac{1}{2}+{\rm i}r;n\bigg)},
\end{gather*}
where
\begin{gather*}\Lambda(\lambda_k)=\frac{1}{\lambda_k+s(s+2n-1)}-\frac{1}{\lambda_k+a(a+2n-1)},
\\
\widetilde{\Lambda}(r)=\frac{1}{\big(s+n-\frac{1}{2}\big)^2+r^2}-
\frac{1}{\big(a+n-\frac{1}{2}\big)^2+r^2}.
\end{gather*}
Therefore, the spectral decomposition of $K(z,w)$ is
\begin{gather*}
K(z,w)
= \sum_{k=0}^{\infty}\Lambda(\lambda_k)u_k(z) \overline{u_k(w)} +\frac{1}{4\pi}\sum_{j=1}^q\int_{-\infty}^{\infty} \widetilde{\Lambda}(r) E_j\bigg(z,\frac{1}{2}+{\rm i}r;n\bigg)\overline{ E_j\bigg(w,\frac{1}{2}+{\rm i}r;n\bigg)}{\rm d}r.
\end{gather*}
Setting $z=w$, we have
\begin{gather*}
K(z,z)= \sum_{k=0}^{\infty}\Lambda(\lambda_k)|u_k(z) |^2+\frac{1}{4\pi}\sum_{j=1}^q\int_{-\infty}^{\infty} \widetilde{\Lambda}(r)\bigg|E_j\bigg(z,\frac{1}{2}+{\rm i}r;n\bigg)\bigg|^2{\rm d}r.
\end{gather*}
Integrating over $X$, one would have
\begin{equation*}%\label{eq0903_7}
\sum_{k=0}^{\infty}\Lambda(\lambda_k)= \iint_X\Biggl(K(z,z)-\frac{1}{4\pi}\sum_{j=1}^q \int_{-\infty}^{\infty} \widetilde{\Lambda}(r)\bigg|E_j\bigg(z,\frac{1}{2}+{\rm i}r;n\bigg) \bigg|^2{\rm d}r \Bigg)y^{2n}{\rm d}\mu(z).
\end{equation*}
As in \cite{Iwaniec2002}, some regularizations are needed to make this integral well-defined. This will be explained in the following.

By definition \eqref{eq0903_2} and \eqref{eq0903_3}, $K(z,w)$ can be written as a series
\begin{gather*}
K(z,w)=\sum_{\gamma\in\Gamma}k(\gamma z,w) \gamma'(z)^n,
\end{gather*}where
\begin{gather*}k(z,w)=\left(\Psi_{n,s}(u(z,w))-\Psi_{n,a}(u(z,w))\right)\frac{(-4)^n}{(z-\bar{w})^{2n}},\end{gather*} with $\Phi_{n,s}(u)$ defined in \eqref{eq0903_1}. Decompose the elements of the group $\Gamma$ into the set that contains only the identity, and the three sets that contain respectively hyperbolic elements, elliptic elements and parabolic elements, we have
\begin{gather*}
K(z,w)=k(z,w)+\sum_{\substack{\gamma\in\Gamma\\\gamma \;\text{is hyperbolic}}}k(\gamma z,w)\gamma'(z)^n
+\sum_{\substack{\gamma\in\Gamma\\\gamma \;\text{is elliptic}}}k(\gamma z,w)\gamma'(z)^n
\\ \hphantom{K(z,w)=}
{}+\sum_{\substack{\gamma\in\Gamma\\\gamma \;\text{is parabolic}}}k(\gamma z,w)\gamma'(z)^n.
\end{gather*}
Therefore,
\begin{gather}\label{eq0903_6}
\sum_{k=0}^{\infty}\Lambda(\lambda_k)=\Xi_0+\Xi_H+\Xi_E+\Xi_P,
\end{gather}
where
\begin{gather*}
\Xi_0=\iint_X k(z,z)y^{2n}{\rm d}\mu(z),
\\
\Xi_H=\iint_X\sum_{\substack{\gamma\in\Gamma\\\gamma \;\text{is hyperbolic}}}k(\gamma z,z)\gamma'(z)^ny^{2n}{\rm d}\mu(z),
\\
\Xi_E=\iint_X \sum_{\substack{\gamma\in\Gamma\\\gamma \;\text{is elliptic}}}k(\gamma z,z)\gamma'(z)^ny^{2n}{\rm d}\mu(z).
\end{gather*}
The term $\Xi_P$ contains the parabolic contribution as well as the absolutely continuous spectrum. We need to do some regularization to make it finite.
Let $F$ be a fundamental domain of $X$ on~$\mathbb{H}$. Given $Y>0$, let
\begin{gather*}
F^Y=F\setminus \bigcup_{j=1}^qF_j^Y,
\end{gather*}
where
\begin{gather*}
F_j^Y=F\cap\sigma_j\big(\{x+{\rm i}y\mid y>Y\}\big).
\end{gather*}
Then we define
\begin{gather*}
\Xi_P=\lim_{Y\rightarrow\infty}\Bigg\{\iint_{F^Y} \sum_{\substack{\gamma\in\Gamma\\\gamma \;\text{is parabolic}}}k(\gamma z,w)\gamma'(z)^ny^{2n}{\rm d}\mu(z)
\\ \hphantom{\Xi_P=\lim_{Y\rightarrow\infty}\Bigg\{}
{}-\frac{1}{4\pi}\sum_{j=1}^q\int_{-\infty}^{\infty}\widetilde{\Lambda}(r)
\iint_{F^Y}\left|E_j\left(z,\frac{1}{2}+{\rm i}r;n\right) \right|^2y^{2n}{\rm d}\mu(z)\,{\rm d}r\Bigg\}.
\end{gather*}
In Appendix \ref{a3}, we show that this limit indeed exist.

 By \eqref{eq0903_13}, we have
\begin{gather*}
k(z,z)y^{2n}=-\frac{1}{4\pi}\bigl\{\psi(s+2n)+\psi(s)-\psi(a+2n)-\psi(a)\bigr\}.
\end{gather*}
Therefore,
\begin{gather*}
\Xi_0=-\frac{|X|}{4\pi}\bigl\{\psi(s+2n)+\psi(s)-\psi(a+2n)-\psi(a)\bigr\},
\end{gather*}
where $|X|$ is the hyperbolic area of $X$ given by
\begin{gather*}
|X|=2\pi\Bigg\{2g-2+q+\sum_{j=1}^v\bigg(1-\frac{1}{m_j}\bigg)\Bigg\}.
\end{gather*}
The computations of $\Xi_H$, $\Xi_E$ and $\Xi_P$ are much more complicated. We leave them to Appendices~\ref{a1},~\ref{a2} and~\ref{a3}, and quote the results here.

For the hyperbolic contributions, let $P$ be the primitive hyperbolic conjugacy classes in $\Gamma$. For a representative $\gamma$ of a primitive hyperbolic class, let $N(\gamma)>1$ be the multiplier of $\gamma$. Then $\Xi_H=\mathscr{E}_H(s)-\mathscr{E}_H(a)$, where
\begin{gather*}
\mathscr{E}_H(s)=\frac{1}{2s+2n-1}\sum_{[\gamma]\in P}\sum_{k=0}^{\infty}\frac{\log N(\gamma)}{N(\gamma)^{s+n+k}-1}.
\end{gather*}

 The elliptic contribution is given by $\Xi_E=\mathscr{E}_E(s)-\mathscr{E}_E(a)$, where
\begin{gather*}
\mathscr{E}_E(s)=\frac{1}{2s+2n-1}\sum_{j=1}^v\sum_{r=0}^{m_j-1} \bigg[\frac{2\alpha_{m_j}(r-n)+1-m_j}{2m_j^2}\psi\bigg(\frac{s+r}{m_j}\bigg)
\\ \hphantom{\mathscr{E}_E(s)=\frac{1}{2s+2n-1}\sum_{j=1}^v\sum_{r=0}^{m_j-1} \bigg[}
{}+\frac{2\alpha_{m_j}(r+n)+1-m_j}{2m_j^2}\psi\bigg(\frac{s+2n+r}{m_j}\bigg)\bigg].
\end{gather*}
If $m$ is a positive integer greater than 1, and $k$ is an integer, $\alpha_m(k)$ is defined to be the least positive residue modulo $m$. Namely, it is the smallest nonnegative integer congruent to~$k$ mo\-dulo~$m$.

Finally, we have $\Xi_P=\mathscr{E}_P(s)-\mathscr{E}_P(a)$, where
\begin{gather*}
\mathscr{E}_P(s)=\frac{1}{(2s+2n-1)^2} \bigg[q-\operatorname{Tr} \Phi\bigg(\frac{1}{2}\bigg) \bigg]+\frac{1}{2}\Sigma(s)
\\ \hphantom{\mathscr{E}_P(s)=}
{}+\frac{q}{2(2s+2n-1)}\bigg\{\psi(s)+\psi(s+2n)-2\log 2-2\psi\bigg(s+n+\frac{1}{2}\bigg)-2\psi(s+n)\bigg\},
\end{gather*}
where
\begin{gather*}
\Sigma(s)=\frac{1}{2\pi}\int_{-\infty}^{\infty}\frac{1}{r^2+\left(s+n-\frac{1}{2}\right)^2}
\frac{\varphi'}{\varphi}\bigg(\frac{1}{2}+{\rm i}r\bigg){\rm d}r.
\end{gather*}

Gathering all the results, we obtain the resolvent trace formula for $\Delta_n$.
\begin{Theorem}[the resolvent trace formula]\label{Resolvent Trace Formula} The resolvent trace formula of the $n$-Laplacian $\Delta_n$ on the Riemann surface $X$ is given by
\begin{gather*}
\mathscr{T}_S(s)-\mathscr{T}_S(a)=\mathscr{T}_G(s)-\mathscr{T}_G(a),
\end{gather*}
where $\mathscr{T}_S(s)$ and $\mathscr{T}_G(s)$ are respectively the spectral trace and the geometric trace given by
\begin{align*}
\mathscr{T}_S(s)=&\sum_{k=0}^{\infty} \frac{1}{\lambda_k+s(s+2n-1)}- \frac{1}{2} \Sigma(s),
\\
\mathscr{T}_G(s)
=&-\frac{|X|}{4\pi}\bigl[\psi(s+2n)+\psi(s)\bigr]+\frac{1}{2s+2n-1}\sum_{[\gamma]\in P}\sum_{k=0}^{\infty}\frac{\log N(\gamma)}{N(\gamma)^{s+n+k}-1}
\\
&+\frac{1}{2s+2n-1}
\sum_{j=1}^v\sum_{r=0}^{m_j-1}\bigg[\frac{2\alpha_{m_j}(r-n)+1-m_j}{2m_j^2} \psi\bigg(\frac{s+r}{m_j}\bigg)
\\
&+\frac{2\alpha_{m_j}(r+n)+1-m_j}{2m_j^2}\psi\bigg(\frac{s+2n+r}{m_j}\bigg)\bigg]
\\
&+\frac{q}{2(2s+2n-1)}\biggl[\psi(s)+\psi(s+2n)-2\log 2 -2\psi\bigg(s+n+\frac{1}{2}\bigg)-2\psi(s+n)\biggr]
\\
&+\frac{A}{(2s+2n-1)^2},
\end{align*}
and $A$ is the constant
\begin{gather*}
A=\operatorname{Tr}\bigg[I-\Phi\bigg(\frac{1}{2}\bigg)\bigg].
\end{gather*}
\end{Theorem}

In \cite{Teo2020}, we have explained that $A$ is an even nonnegative integer. In applying the resolvent trace formula, we always take $a$ to be a positive constant that is large enough.

As mentioned in the introduction, $\Delta_n$ has the same spectrum as $D_n+n(n-1)$. Therefore,
$\Delta_n+s(s+2n-1)$ has the same spectrum as
\begin{gather*}
D_n+s(s+2n-1)+n(n-1)=D_n+(s+n)(s+n-1).
\end{gather*}
Hence, if we replace the $s$ in the resolvent trace formula for $D_n$ that were derived in \cite{Fischer1987, Hejhal1983} by~$s+n$, one should get the resolvent trace formula for $\Delta_n$ obtained in Theorem \ref{Resolvent Trace Formula}.

In principle, we can directly quote the formulas in \cite{Fischer1987, Hejhal1983} and proceed directly to derive the regularized determinants of $n$-Laplacians.
However, we have found some inconsistencies in the results in these two references. One of the goals of this work is to resolve this inconsistency, giving a concise reference to the computations that lead to the resolvent trace formula, supplementing the $n=0$ case done in the book \cite{Iwaniec2002}. When doing so, we find that it is more natural to consi\-der~$\Delta_n$ which is a matter of taste. The use of $\Delta_n+s(s+2n-1)$ instead of $\Delta_n+(s-n)(s+n-1)$ is so that setting $s=0$ in $\Delta_n+s(s+2n-1)$ give $\Delta_n$.

\section{Dimensions of spaces holomorphic differentials}

For $n\geq 1$, the space of holomorphic $n$-differentials is the subspace of $H_n^2(\Gamma)$ which are holomorphic. Hence, they are eigenvectors of $\Delta_n$ with eigenvalue 0. We also call them zero modes or holomorphic cusp forms of weight $2n$.

Let $d_n$ be the dimension of the space of holomorphic $n$-differentials of $X$. Although the formula for $d_n$ is well-known, we would like to show how this can be derived from the resolvent trace formula, similar to what have been done in \cite{Hejhal1983}.

When $n=0$, $d_0$ is the dimension of space of holomorphic functions, which is equal to 1, corresponding to the constant functions.

Notice that
the hyperbolic contribution to the resolvent trace formula can be written as
\begin{gather*}
\Xi_H= \frac{1}{2s+2n-1}\frac{\rm d}{{\rm d}s}\log Z(s+n)-\frac{1}{2a+2n-1}\frac{\rm d}{{\rm d}s}\log Z(a+n),
\end{gather*} where $Z(s)$ is the Selberg zeta function of $X$ defined as
\begin{gather*}
Z(s)=\prod_{[\gamma]\in P}\prod_{k=0}^{\infty}\big(1-N(\gamma)^{-s-k}\big).
\end{gather*}

It is well known that $Z(s)$ has a zero of order 1 at $s=1$ and it is absolutely convergent when $\operatorname{Re}s>1$. Therefore, when $n=1$, $\tfrac{\rm d}{{\rm d}s}\log Z(s+n)$ has residue 1 at $s=0$. When $n\geq 2$, $\tfrac{\rm d}{{\rm d}s}\log Z(s+n)$ has residue 0 at $s=0$.

As we explained above, the dimension of the space of holomorphic $n$-differentials $d_n$ is equal to the multiplicity of $0$ as a discrete eigenvalue of $\Delta_n$. According to Theorem~\ref{polescatter}, the residue of $\Sigma(s)$ at $s=0$ is 0. Hence, $d_n$ is equal to $(2n-1)$ times the residue at $s=0$ of the function on the left hand side of the resolvent trace formula. To determine the value of $d_n$ from the resolvent trace formula, we need to find the residue at $s=0$ of the function at the right-hand side of the trace formula.
It then amounts to understanding the zeros and poles of $Z(s)$ at positive integers, as well as the residues of $\psi(s)$ at rational numbers.

Since \begin{gather*}\psi(s)=-\gamma+\sum_{k=0}^{\infty}\left(\frac{1}{k+1}-\frac{1}{s+k}\right),\end{gather*} we find that the residues of $\psi(s)$ at $s=0, -1, -2, \dots$ are all equal to $-1$. $\psi(s)$ does not have poles at other points.

 Using the fact that $\alpha_{m_j}(-1)=m_j-1$, we find that
\begin{gather*}
d_1=\frac{1}{2}\Bigg[2g-2+q+\sum_{j=1}^v\bigg(1-\frac{1}{m_j}\bigg)\Bigg] +1-\frac{1}{2}\sum_{j=1}^v\bigg(1-\frac{1}{m_j}\bigg)-\frac{q}{2}=g,
\end{gather*}
which is a well-known result, since the space of holomorphic one-differentials is exactly the space of abelian differentials.

When $n\geq 2$,
\begin{gather*}
d_n=\frac{2n-1}{2}\Bigg[2g-2+q+\sum_{j=1}^v\bigg(1-\frac{1}{m_j}\bigg)\Bigg]
+\frac{1}{2}\sum_{j=1}^v \frac{m_j-1-2\alpha_{m_j}(-n)}{m_j}-\frac{q}{2}.
\end{gather*}
Since
\begin{gather*}
\frac{\alpha_{m_j}(-n)}{m_j}=-\frac{n}{m_j}-\bigg\lfloor{-}\frac{n}{m_j}\bigg\rfloor,
\end{gather*}
we find that when $n\geq 2$,
\begin{gather*}
d_n= (2n-1)(g-1)+(n-1)q+\sum_{j=1}^v\bigg( \bigg\lfloor n-\frac{n}{m_j}\bigg\rfloor\bigg).
\end{gather*}

\section[The determinant of n-Laplacian]{The determinant of $\boldsymbol n$-Laplacian}

In this section, we want to derive our main result -- the formula for the determinant of $n$-Laplacian
$\Delta_n$ for a cofinite Riemann surface $X$ in terms of the Selberg zeta function for the Riemann surface. This extends our result in \cite{Teo2020} to the case where $n\geq 1$. The results are not entirely new. For compact hyperbolic surfaces, the relation between the determinant of $\Delta_n$ and the Selberg zeta function have been obtained by D'Hoker and Phong \cite{DHokerPhong1986} and Sarnak \cite{Sarnak1987}. In~\cite{Efrat1988}, Efrat considered the determinant of $\Delta_0+s(s-1)$ for cofinite hyperbolic surfaces without elliptic points. For congruence subgroups $\Gamma_0(N)$, $\Gamma_1(N)$ and $\Gamma(N)$, Koyama has obtained the relation in his work \cite{Koyama1991, Koyama1991_2}. Gong has considered the more general case of Laplacian operators on automorphic forms of nonzero weights in \cite{Gong1995}, but this work has not attracted much attention.

We would first derive the determinant of $\Delta_n+s(s+2n-1)$.
Since $\Delta_n$ contains absolutely continuous spectrum, we need to be careful with defining the determinant. Following \cite{Efrat1988,Koyama1991_2, VenkovKalininFaddeev}, let
\begin{gather*}
\zeta(w,s)= \sum_{k=0}^{\infty}\frac{1}{(\lambda_k+s(s+2n-1))^w} -\frac{1}{4\pi}\int_{-\infty}^{\infty}
\frac{1}{\big[\big(s+n-\tfrac{1}{2}\big)^2+r^2 \big]^w} \frac{\varphi'}{\varphi}\bigg(\frac{1}{2}+{\rm i}r\bigg){\rm d}r
\end{gather*}
be the spectral zeta function of $X$ correspond to $\Delta_n+s(s+2n-1)$. This expression is well-defi\-ned when $\operatorname{Re}w$ is large enough.
 It can be analytically continued to a neighbourhood of~$w=0$. The~zeta regularized determinant $\det (\Delta_n+s(s+2n-1))$ is defined as
\begin{gather*}%\label{eq_det}
\det ((\Delta_n+s(s+2n-1))=\exp\left(-\zeta_w(0,s)\right).
\end{gather*}

By uniqueness of analytic continuation, we find that
\begin{gather*}%\label{eq1}
\frac{\rm d}{{\rm d}s}\frac{1}{2s+2n-1}\frac{\rm d}{{\rm d}s}\log \det (\Delta_n+s(s+2n-1))
= \frac{\rm d}{{\rm d}s}\frac{1}{2s+2n-1}\frac{\rm d}{{\rm d}s}(-\zeta_w(0,s))
\\ \qquad
{}= \frac{\rm d}{{\rm d}s}\Bigg\{\sum_{k=0}^{\infty}\bigg[\frac{1}{\lambda_k+s(s+2n-1)} -\frac{1}{\lambda_k+a(a+2n-1)}\bigg]
\\ \qquad\hphantom{= \frac{\rm d}{{\rm d}s}\Bigg\{}
{}-\frac{1}{4\pi}\int_{-\infty}^{\infty}
\bigg[\frac{1}{\big(s+n-\tfrac{1}{2}\big)^2+r^2 }-\frac{1}{\big(a+n-\tfrac{1}{2}\big)^2+r^2 }\bigg]\frac{\varphi'}{\varphi}\bigg(\frac{1}{2}+{\rm i}r\bigg){\rm d}r\Bigg\}.
\end{gather*}
To compute this using the resolvent trace formula, we first recall that the Alekseevskii--Barnes double gamma function $\Gamma_2(s)$ is defined as \cite{Alekseevskii1889, Barnes}:
\begin{gather*}
\Gamma_2(s+1)=\frac{1}{(2\pi)^{\frac{s}{2}}}{\rm e}^{\frac{s}{2}+\frac{\gamma+1}{2}s^2} \prod_{k=1}^{\infty}\bigg(1+\frac{s}{k}\bigg)^{-k}{\rm e}^{s-\frac{s^2}{2k} }.
\end{gather*}It satisfies the equation
\begin{gather*}
\Gamma_2(s+1)=\frac{\Gamma_2(s)}{\Gamma(s)}.
\end{gather*}
Using
\begin{gather*}
\frac{\rm d}{{\rm d}s}\log\Gamma_2(s+1)=-\frac{1}{2}\log(2\pi)+\frac{1}{2}+(\gamma+1)s -s\sum_{k=1}^{\infty}\bigg(\frac{1}{k}-\frac{1}{s+k}\bigg),
\end{gather*}and
\begin{gather*}\psi(s)=\frac{\rm d}{{\rm d}s}\log\Gamma(s) =-\gamma-\frac{1}{s}-\sum_{k=1}^{\infty}\bigg(\frac{1}{s+k}-\frac{1}{k}\bigg),
\end{gather*}
we find that
\begin{gather*}
\frac{\rm d}{{\rm d}s}\log\frac{(2\pi)^{2s+2n-1} \Gamma_2(s+2n)^2\Gamma_2(s)^2 \Gamma(s+2n)^{2n-1}}{\Gamma(s)^{2n+1}}
\\ \qquad
{}=-(2s+2n-1)(\psi(s+2n)+\psi(s)-2).
\end{gather*}
Hence,
the resolvent trace formula says that
\begin{gather*}
\frac{\rm d}{{\rm d}s}\frac{1}{2s+2n-1}\frac{\rm d}{{\rm d}s}\log\det(\Delta_n+s(s+2n-1))
\\ \qquad
{}=\frac{\rm d}{{\rm d}s}\frac{1}{2s+2n-1}\frac{\rm d}{{\rm d}s}\log\big[Z_{\infty}(s)Z(s+n)Z_{{\rm ell}(s)}Z_{{\rm par}}(s)\big],
\end{gather*}
where
\begin{align}
Z_{\infty}(s)&=\bigg[\frac{(2\pi)^{2s+2n-1} \Gamma_2(s+2n)^2\Gamma_2(s)^2 \Gamma(s+2n)^{2n-1}}{\Gamma(s)^{2n+1}}\bigg]^{\frac{|X|}{4\pi}},\nonumber
\\
Z_{{\rm ell}}(s)&=\prod_{j=1}^v \prod_{r=0}^{m_j-1}\Gamma\bigg(\frac{s+r}{m_j}\bigg)^{\frac{2\alpha_{m_j}(r-n)+1-m_j}{2m_j}}\Gamma \bigg(\frac{s+2n+r}{m_j}\bigg)^{\frac{2\alpha_{m_j}(r+n)+1-m_j}{2m_j}}\nonumber
\\
&=\prod_{j=1}^v\frac{\prod_{r=0}^{m_j-1}\Gamma \big(\frac{s+r}{m_j}\big)^{\frac{\alpha_{m_j}(r-n)}{m_j}}\Gamma \big(\frac{s+2n+r}{m_j}\big)^{\frac{\alpha_{m_j}(r+n)}{m_j}}}
{\bigl[(2\pi)^{m_j-1}m_j^{-(2s+2n-1)}\Gamma(s+2n)\Gamma(s)\bigr]^{\frac{1}{2}\big(1-\frac{1}{m_j}\big)}}, \nonumber
\\
Z_{{\rm par}}(s)&=\bigg[\frac{\Gamma(s)\Gamma(s+2n)}{2^{2s+2n-1}\Gamma(s+n)^2 \Gamma\big(s+n+\frac{1}{2}\big)^2}\bigg]^{q/2}\bigg(s+n-\frac{1}{2}\bigg)^{\frac{A}{2}}.
\label{various}
\end{align}
It follows that
\begin{gather*}
\det(\Delta_n+s(s+2n-1))= Z_{\infty}(s)Z(s+n)Z_{{\rm ell}(s)}Z_{{\rm par}}(s){\rm e}^{B\left(s+n-\frac{1}{2}\right)^2+D}
\end{gather*}
for some constants $B$ and $D$. To determine $B$ and $D$, we need to study the behavior of both sides when $s\rightarrow \infty$.

First we define the heat kernel
\begin{gather*}%\label{eq0903_9}
\theta(t)= \sum_{k=0}^{\infty}{\rm e}^{-t\lambda_k} -\frac{1}{4\pi}\int_{-\infty}^{\infty}{\rm e}^{-t\left(r^2+\left[n-\frac{1}{2}\right]^2\right)}\frac{\varphi'}{\varphi}\bigg(\frac{1}{2}+{\rm i}r\bigg){\rm d}r
\end{gather*}
for $t>0$. Notice that
\begin{gather*}
\zeta(w,s)=\frac{1}{\Gamma(w)}\int_0^{\infty}t^{w-1}\theta(t){\rm e}^{-t(s(s+2n-1))}{\rm d}t.
\end{gather*}
Although we can derive the explicit formula for $\theta(t)$ using the general trace formula, but we do not need it. We only need the asymptotic expansion of $\theta(t)$ as $t\rightarrow 0^+$. By comparison to the $n=0$ or the general theory about heat kernel, we know that the asymptotic expansion of $\theta(t)$ has the form
\begin{gather*}
\theta(t)= \bigg(\frac{b}{t}+\frac{c\log t}{\sqrt{t}}+\frac{d}{\sqrt{t}}+h\bigg){\rm e}^{-t\left(n-\frac{1}{2}\right)^2}+O\big(\sqrt{t}\big)
\end{gather*}
for some constants $b$, $c$, $d$ and $h$.
Let \begin{equation*}%\label{eq0903_5}
u=s+n-\frac{1}{2}.
\end{equation*}
Then as $u\rightarrow\infty$,
\begin{gather*}
\log\det(\Delta_n+s(s+2n-1))\sim -\frac{\pa}{\pa w}\bigg|_{w=0} \frac{1}{\Gamma(w)}\int_0^{\infty}t^{w-1}\bigg(\frac{b}{t}+\frac{c\log t}{\sqrt{t}}
+\frac{d}{\sqrt{t}}+h\bigg){\rm e}^{-tu^2}{\rm d}t.
\end{gather*}
Using
 \begin{gather*}
\frac{\pa}{\pa w}\bigg|_{w=0} \frac{1}{\Gamma(w)}\int_0^{\infty}t^{w-1}{\rm e}^{-tu^2}{\rm d}t= -2\log u,
\\
\frac{\pa}{\pa w}\bigg|_{w=0} \frac{1}{\Gamma(w)}\int_0^{\infty}t^{w-2}{\rm e}^{-tu^2}{\rm d}t=-u^2+2u^2\log u,
\\
\frac{\pa}{\pa w}\bigg|_{w=0}\frac{1}{\Gamma(w)}\int_0^{\infty}t^{w-\frac{3}{2}}{\rm e}^{-tu^2}{\rm d}t=-2\sqrt{\pi}u,
\\
\frac{\pa}{\pa w}\bigg|_{w=0} \frac{1}{\Gamma(w)}\int_0^{\infty}t^{w-\frac{3}{2}}\log t \;{\rm e}^{-tu^2}{\rm d}t=-2\sqrt{\pi}u(2-2\log 2-\gamma-2\log u),
\end{gather*}
we find that as $u\rightarrow\infty$,
\begin{gather*}
\log\det(\Delta_n+s(s+2n-1))=bu^2-2bu^2\log u +2c\sqrt{\pi}u(2-2\log 2-\gamma-2\log u)
\\ \hphantom{\log\det(\Delta_n+s(s+2n-1))=}
{}+2d\sqrt{\pi}u+2h\log u+o(1).
\end{gather*}
On the other hand, it is obvious from definition that $\log Z(s+n)$ is $o(1)$ as $s\rightarrow \infty$. From the asymptotic behaviors
\begin{gather}
\log\Gamma(s)
= \bigg(s-\frac{1}{2}\bigg) \log s - s +\frac{1}{2}\log2\pi+o(1),\nonumber
\\
\log\Gamma_2(s+1)
= -\frac{1}{2}s^2\log s+\frac{3}{4}s^2 -\frac{s}{2}\log(2\pi)+\frac{1}{12}\log s-\zeta'(-1)+o(1),
\label{gammabe}
\end{gather}
we can deduce the following.
As $s\rightarrow\infty$,
\begin{gather*}
\log Z_{\infty}(s)=\frac{|X|}{4\pi}\bigg\{\bigg({-}2u^2+2n^2-\frac{1}{6}\bigg)\log u +3u^2-4\zeta'(-1) \bigg\}+o(1),
\\
\log Z_{{\rm par}}(s)=\frac{q}{2}\{-(2u+1)\log u +2u-\log(2\pi)-2u\log 2\}+\frac{A}{2}\log u+o(1).
\end{gather*}
The asymptotic behavior of $Z_{{\rm ell}}(s)$ is the most complicated one.
We leave the computation to Appendix~\ref{Elliptic} and quote the result here. We find that as $s\rightarrow\infty$,
\begin{gather*}
\log Z_{{\rm ell}}(s)=\mathscr{B}\log u+\mathscr{D}+o(1),
\end{gather*}
where
\begin{gather*}
\mathscr{B}= \sum_{j=1}^v\bigg(\frac{m_j^2-1}{6m_j}-\frac{\alpha_{m_j}(n)(m_j-\alpha_{m_j}(n))}{m_j}\bigg),
\\
\mathscr{D}= -\sum_{j=1}^v\bigg(\frac{m_j^2-1}{6m_j}-\frac{\alpha_{m_j}(n)(m_j-\alpha_{m_j}(n))}{m_j}\bigg)\log m_j.
\end{gather*}
{\samepage
Hence, we find that as $s\rightarrow \infty$,
\begin{gather*}
\log\det(\Delta_n+s(2+2n-1))
\\ \qquad
{}= \frac{|X|}{4\pi}\bigg\{\bigg({-}2u^2+2n^2-\frac{1}{6}\bigg)\log u +3u^2-4\zeta'(-1) \bigg\} +\mathscr{B}\log u+\mathscr{D}
\\ \qquad\phantom{=}
{}+\frac{q}{2}\big\{{-}(2u+1)\log u +2u-\log(2\pi)-2u\log 2\big\}+\frac{A}{2}\log u+Bu^2+D+o(1)
\\ \qquad
{}=bu^2-2bu^2\log u+2c\sqrt{\pi}u(2-2\log 2-\gamma-2\log u) +2d\sqrt{\pi}u+2h\log u+o(1).
\end{gather*}

}\noindent
Comparing the two asymptotic expansions give
\begin{gather}
B= -\frac{|X|}{2\pi},\nonumber
\\
D= \frac{|X|}{\pi}\zeta'(-1)+\frac{q}{2}\log(2\pi) +\sum_{j=1}^v \bigg(\frac{m_j^2-1}{6m_j}-\frac{\alpha_{m_j}(n)(m_j-\alpha_{m_j}(n))}{m_j}\bigg)\log m_j.
 \label{constant}
\end{gather}
This gives the following result.
\begin{Theorem}\label{det2}
Let $n$ be a nonnegative integer and let $\Delta_n$ be the $n$-Laplacian of the Riemann surface $X$. Then
\begin{gather*}
\det\left(\Delta_n+s(s+2n-1)\right)=Z_{\infty}(s)Z(s+n)Z_{{\rm ell}(s)}Z_{{\rm par}}(s){\rm e}^{B\left(s+n-\frac{1}{2}\right)^2+D},
\end{gather*}where
$Z(s)$ is the Selberg zeta function, $Z_{\infty}(s)$, $Z_{{\rm ell}}(s)$ and $Z_{{\rm par}}(s)$ are defined in \eqref{various}, and the constants $B$ and $D$ are given by
\eqref{constant}.\end{Theorem}

Since $\Delta_n$ has zero eigenvalues, we need to remove these zero modes when we define the regularized determininant of $\Delta_n$. When $n\geq 2$, $\varphi(s+n)$ is regular when $s=0$. Hence, a~rea\-so\-nable definition is
\begin{gather}\label{eq314_2}
\det\!^{\prime}\Delta_n=\lim_{s\rightarrow 0}\frac{\det(\Delta_n+s(s+2n-1))}{[s(s+2n-1)]^{d_n}}=\frac{1}{(2n-1)^{d_n}}\lim_{s\rightarrow 0}\frac{\det(\Delta_n+s(s+2n-1))}{s^{d_n}},
\end{gather}
where $d_n$ is the dimension of holomorphic $n$-differentials.
When $n=1$, Proposition \ref{polescatter} shows that we can also use \eqref{eq314_2} to define $\det'\Delta_1$.

The $n=0$ case is more complicated. As discussed in \cite{Venkov1982}, the possible zero of $\varphi(s)$ at $s=0$ would give some extra contribution. If $n_0$ is the order of zero of $\varphi(s)$ at $s=0$, then we should define
\begin{gather*}%\label{eq314_3}
\det\!^{\prime}\Delta_0= \lim_{s\rightarrow 0}\frac{\det(\Delta_0+s(s-1))}{[s(s -1)]^{1-n_0}} = (-1)^{n_0-1}\lim_{s\rightarrow 0}\frac{\det(\Delta_0+s(s-1))}{s^{1-n_0}}.
\end{gather*}

Now we want to derive the formula for $\det\!^{\prime}\Delta_n$ from Theorem \ref{det2}. We discuss the case $n=0$ and $n\geq 1$ separately.
When $n=0$, we find that when $s\rightarrow 0$,
\begin{gather*}
Z_{\infty}(s)\sim \bigg(\frac{1}{2\pi}\bigg)^{\frac{|X|}{4\pi}}s^{-\frac{|X|}{2\pi}},
\\
Z_{{\rm ell}}(s)\sim \prod_{j=1}^v s^{ \frac{m_j-1}{m_j} }m_j^{ \frac{1-m_j}{m_j}} \prod_{r=1}^{m_j-1}\Gamma\bigg(\frac{r}{m_j}\bigg)^{\frac{2r+1-m_j}{m_j}},
\\
Z_{{\rm par}}(s)\sim \bigg({-}\frac{1}{2}\bigg)^{\frac{A}{2}}\bigg(\frac{2}{\pi}\bigg)^{\frac{q}{2}}.
\end{gather*}These imply that
\begin{gather*}\det\!^{\prime}\Delta_0=\mathcal{C}_0\lim_{s\rightarrow \infty}\frac{Z(s)}{s^{2g-1+q-n_0}},
\end{gather*}
where
\begin{gather}\label{c0}
\mathcal{C}_0=(-1)^{\frac{A}{2}+1-n_0} 2^{ q-\frac{A}{2}}(2\pi)^{-\frac{q}{2}-\frac{|X|}{4\pi} }{\rm e}^{\frac{B}{4}+D}\prod_{j=1}^vm_j^{ \frac{1-m_j}{m_j}}\prod_{r=1}^{m_j-1}\Gamma\bigg(\frac{r}{m_j}\bigg)^{\frac{2r+1-m_j }{m_j}}.
\end{gather}

When $n\geq 1$, we find that as $s\rightarrow 0$,
\begin{gather*}
Z_{\infty}(s)\sim s^{(2n-1)\frac{|X|}{4\pi}} \bigl[(2\pi)^{2n-1}\Gamma_2(2n)^2\Gamma(2n)^{2n-1}\bigr]^{\frac{|X|}{4\pi}},
\\
Z_{{\rm ell}}(s)\sim \prod_{j=1}^v\Bigg\{s^{ \frac{m_j-1-2\alpha_{m_j}(-n)}{2m_j}}m_j^{ \frac{ 2\alpha_{m_j}(-n)+1-m_j}{2m_j}}
\\ \hphantom{Z_{{\rm ell}}(s)\sim \prod_{j=1}^v\Bigg\{}
{}\times
\prod_{r=1}^{m_j-1}\Gamma\bigg(\frac{r}{m_j}\bigg)^{\frac{2\alpha_{m_j}(r-n)+1-m_j}{2m_j}} \prod_{r=0}^{m_j-1}\Gamma\bigg(\frac{2n+r}{m_j}\bigg)^{\frac{2\alpha_{m_j}(r+n)+1-m_j}{2m_j}}\Bigg\},
\\
Z_{{\rm par}}(s)\sim s^{-\frac{q}{2}}\bigg[\frac{\Gamma(2n)}{2^{2n-1}\Gamma(n)^2 \Gamma\big(n+\frac{1}{2}\big)^2}\bigg]^{q/2}\bigg(n-\frac{1}{2}\bigg)^{\frac{A}{2}}.
\end{gather*}
Using also the fact that $Z(s)$ has a zero of order 1 at $s=1$, and $Z(n)$ is nonzero, we conclude that when $n=1$,
\begin{gather*}
\det\!^{\prime}\Delta_1=\mathcal{C}_1Z'(1),
\end{gather*}
while when $n\geq 2$,
\begin{gather*}
\det\!^{\prime}\Delta_n=\mathcal{C}_nZ(n).
\end{gather*}
The constant $\mathcal{C}_n$ is given by
\begin{gather}
\mathcal{C}_n= \bigl[ (2\pi)^{2n-1}\Gamma_2(2n)^2\Gamma(2n)^{2n-1}\bigr]^{\frac{|X|}{4\pi}}\nonumber
\\ \hphantom{\mathcal{C}_n=}
{}\times\prod_{j=1}^v\Bigg\{m_j^{ \frac{ 2\alpha_{m_j}(-n)+1-m_j}{2m_j}}
 \prod_{r=1}^{m_j-1}\Gamma\bigg(\frac{r}{m_j}\bigg)^{\frac{2\alpha_{m_j}(r-n)+1-m_j}{2m_j}} \prod_{r=0}^{m_j-1}\Gamma\bigg(\frac{2n+r}{m_j}\bigg)^{\frac{2\alpha_{m_j}(r+n)+1-m_j}{2m_j}}\Bigg\}
\nonumber
\\ \hphantom{\mathcal{C}_n=}
\times \bigg[\frac{2^{2n-1}}{\pi\Gamma(2n)}\bigg]^{q/2}(2n-1)^{-d_n} \bigg(n-\frac{1}{2}\bigg)^{\frac{A}{2}}{\rm e}^{B\left(n-\frac{1}{2}\right)^2+D}.\label{cn}
\end{gather}

Finally, we obtain the main result of our paper.
\begin{Theorem}\label{determinant}
When $n\geq 0$, the regularized determinant of the $n$-Laplacian $\Delta_n$ of $X$ is given~by \begin{gather*}
\det\!^{\prime}\Delta_n=\begin{cases}
\mathcal{C}_0Z_0, &n=0,
\\[.5ex]
\mathcal{C}_1Z'(1), &n=1,
\\[.5ex]
\mathcal{C}_nZ(n), &n\geq 2,
\end{cases}
\end{gather*}
where
\begin{gather*}
Z_0=\lim_{s\rightarrow 0}\frac{Z(s)}{s^{2g-1+q-n_0}},
\end{gather*}
$\mathcal{C}_0$ is given by \eqref{c0} and for $n\geq 1$, $\mathcal{C}_n$ is given by \eqref{cn}.
\end{Theorem}

This establishes the exact relation between the regularized determinant $\det\!'\Delta_n$ and the Selberg zeta function. Notice that the constant $\mathcal{C}_n$ only depends on the type of the Riemann surface and is independent of the moduli. In \cite{TakhtajanZograf1991} and \cite{TakhtajanZograf2019}, Takhtajan and Zograf established the local index theorem for Riemann surfaces with cusps and with ramification points, using $Z'(1)$ and $Z(n)$, $n\geq 2$, as defintions for $\det\!'\Delta_1$ and $\det\!'\Delta_n$, $n\geq 2$ respectively. Since they only considered the second variation of $\log\det\!'\Delta_n$ on the moduli space, Theorem \ref{determinant} justifies their approach. However, if one wants to consider the holomorphic factorization of the determinant of Laplacian, as considered in \cite{McintyreTakhtajan2006} for compact Riemann surfaces, the precise value of $\mathcal{C}_n$ becomes important. They would also be important for the extension of the work \cite{MontpletPippich2020} by Freixas i Montplet and von Pippich.\looseness=-1

We would like to point out that our definition of the regularized determinant of Laplacian has taken into account the contribution from the absolutely continuous spectrum. For a generic cofinite Riemann surface with cusps, it is not known whether it has an infinite discrete spectrum. In fact, it has been conjectured to be the opposite. The inclusion of the absolutely continuous spectrum is essential to render the determinant to be equal to a moduli independent constant times $Z(n)$ (for $n\geq 2$).

In the resolvent trace formula, the contribution from the absolutely continuous spectrum depends on $\varphi(s)$, the determinant of the scattering matrix, which is moduli dependent. It~would be interesting if one can compute the contribution of this term in the local index theorem explicitly.

\appendix
 \section{The inversion formulas}\label{ainversion}

 In this appendix, we want to derive a useful formula that is needed in the computation of the elliptic contribution to the trace equation~\eqref{eq0903_6} in Appendix~\ref{a2}.

 When
 \begin{gather*}
 s=\frac{1}{2}+{\rm i}r,
 \end{gather*}
 we find that
\begin{gather*}
-(s-n)(s+n-1)= r^2+\bigg(n-\frac{1}{2}\bigg)^2.
\end{gather*}
Given a function $\mathscr{H}(\lambda)$, define
\begin{gather*}
h(r)=\mathscr{H}\bigg(r^2+\bigg(n-\frac{1}{2}\bigg)^2\bigg).
\end{gather*}
In this work, we let $\mathscr{H}(\lambda)$ be a function such that the function $h(r)$ satisfies the following conditions:
\begin{itemize}\itemsep=0pt
\item $h(r)$ is holomorphic in the strip $|\operatorname{Im}(r)|\leq\frac{1}{2}+\delta$ for some $\delta>0$,
\item there exist positive constants $M$ and $\varepsilon$ such that $ |h(r)|\leq \frac{M}{(1+|r|)^{2+\varepsilon}}$ in the strip.
\end{itemize}
Using Proposition \ref{P1} with $k(z,w)$ the kernel of the operator $\mathscr{H}(\Delta_n)$, we find that
\begin{gather*}
\mathscr{H}(-(s-n)(s+n-1))
\\ \qquad
{}=\int_0^{\infty}\int_{-\infty}^{\infty}\Psi\bigg(\frac{x^2+(y-1)^2}{4y}\bigg) \frac{(-4)^n}{(x-{\rm i}(y+1))^{2n}}y^{s+n}\frac{{\rm d}x\,{\rm d}y}{y^2}
\\ \qquad
{}=2(-1)^n\int_0^{\infty}\int_{-\infty}^{\infty} \Psi\bigg(x^2+\frac{(y-1)^2}{4y}\bigg)\frac{1}{\bigg(x-{\rm i}\displaystyle\frac{(y+1)}{2\sqrt{y}}\bigg)^{2n}} y^{s-\frac{3}{2}}{\rm d}x\,{\rm d}y.
\end{gather*}
Let\vspace{-1ex}
\begin{gather*}
Q(v)=2(-1)^n\int_{-\infty}^{\infty} \Psi(x^2+v)\frac{1}{\big(x-{\rm i}\sqrt{v+1}\big)^{2n}} {\rm d}x
\\ \hphantom{Q(v)}
{}=2(-1)^n\int_{0}^{\infty} \Psi(x^2+v)\frac{\big(x+{\rm i}\sqrt{v+1}\big)^{2n}+\big(x-{\rm i}\sqrt{v+1}\big)^{2n}}{\big(x^2+1+v\big)^{2n}} {\rm d}x.
\end{gather*}
Then\vspace{-1ex}
\begin{gather*}
h(r)= \int_0^{\infty}Q\bigg(\frac{(y-1)^2}{4y}\bigg)y^{{\rm i}r-1}{\rm d}y= \int_{-\infty}^{\infty}Q\bigg(\sinh^2\frac{t}{2}\bigg) {\rm e}^{{\rm i}tr} {\rm d}t.
\end{gather*}
Define\vspace{-1ex}
\begin{gather*}
g(t)=Q\bigg(\sinh^2\frac{t}{2}\bigg).
\end{gather*}
Then\vspace{-1ex}
\begin{gather*}
h(r)=\int_{-\infty}^{\infty} g(t){\rm e}^{{\rm i}tr}{\rm d}t.
\end{gather*}
Namely, $h(r)$ is the Fourier transform of $g(t)$.
The theory of Fourier transform implies that\vspace{-1ex}
\begin{gather*}
g(t)=\frac{1}{2\pi}\int_{-\infty}^{\infty} h(r) {\rm e}^{-{\rm i}rt}{\rm d}r.
\end{gather*}
When\vspace{-1ex}
\begin{gather*}
\Psi(u)=\Psi_{n,s}(u)-\Psi_{n,a}(u),
\end{gather*}
where $\Psi_{n,s}(u)$ is defined in \eqref{eq0903_1}, we have\vspace{-1ex}
\begin{align*}
h(r)&=\frac{1}{r^2+\left(n-\tfrac{1}{2}\right)^2+s(s+2n-1)}-\frac{1}{r^2+\left(n-\tfrac{1}{2}\right)^2+a(a+2n-1)}
\\
&=\frac{1}{r^2+\left(s+n-\tfrac{1}{2}\right)^2 }-\frac{1}{r^2+\left(a+n-\tfrac{1}{2}\right)^2 }.
\end{align*}
It is well-known that $h(r)$ is the Fourier transform of the function
\begin{gather*}
g(t)=\frac{1}{2s+2n-1}{\rm e}^{-t\left(s+n-\frac{1}{2}\right)}-\frac{1}{2a+2n-1}{\rm e}^{-t\left(a+n-\frac{1}{2}\right)}.
\end{gather*}
Next we want to express $\Psi(u)$ in terms of $Q(v)$ in the particular case where $\Psi(u)=\Psi_{n,s}(u)-\Psi_{n,a}(u)$.
Inspired by the formula in \cite{Hejhal1976}, we claim that
\begin{align*}
 \Psi(x) &=-\frac{1}{2\pi}\int_{-\infty}^{\infty} Q'\big(x+t^2\big) \big(\sqrt{x+1+t^2}-t\big)^{2n}{\rm d}t
 \\
 &=-\frac{1}{2\pi}\int_0^{\infty} Q'\big(x+t^2\big)\left[\big(\sqrt{x+1+t^2}-t\big)^{2n}+\big(\sqrt{x+1+t^2}+t\big)^{2n}\right]{\rm d}t.
 \end{align*}
Since $\Psi_{n,s}(u)$ is a linear combination of terms of the form
\begin{gather*}
\frac{1}{(u+1)^{\alpha}},
\end{gather*}
it suffices to consider the case where
\begin{gather*}
\Psi(x)=\frac{1}{(x+1)^{\alpha}}.
\end{gather*}
 In this case,
\begin{align*}
Q(v)&=4(-1)^n \int_{0}^{\infty}\frac{1}{(x^2+v+1)^{\alpha+2n}}\sum_{r=0}^n
\begin{pmatrix} 2n\\2r\end{pmatrix}
(v+1)^{n-r}(-1)^{n-r} x^{2r}{\rm d}x
\\ \hphantom{}
&=2\sum_{r=0}^n \begin{pmatrix} 2n\\2r\end{pmatrix}
(v+1)^{n-r}(-1)^{r} \int_0^{\infty}\frac{x^{r-\frac{1}{2}}}{(x+v+1)^{\alpha+2n}}{\rm d}x
\\
&=\frac{2}{(v+1)^{\alpha+n-\frac{1}{2}}}\sum_{r=0}^n
\begin{pmatrix} 2n\\2r\end{pmatrix}
(-1)^r\frac{\Gamma\left(r+\frac{1}{2}\right)\Gamma\left(\alpha+2n-r-\frac{1}{2}\right)}{\Gamma(\alpha+2n)}.
\end{align*}
Some manipulations give
\begin{align*}
Q(v)
&= \frac{2}{(v+1)^{\alpha+n-\frac{1}{2}}}\frac{\Gamma\left(\alpha+2n -\frac{1}{2}\right)}{\Gamma(\alpha+2n)}\Gamma\left(\frac{1}{2}\right)
\sum_{r=0}^n\frac{(-n)_r\left(-n+\frac{1}{2}\right)_r}{\left(\frac{1}{2}\right)_r r!}\frac{\left(\frac{1}{2}\right)_r}{\left(-\alpha-2n+\frac{3}{2}\right)_r}
\\
&=\frac{2\sqrt{\pi}}{(v+1)^{\alpha+n-\frac{1}{2}}}\frac{\Gamma\left(\alpha+2n -\frac{1}{2}\right)}{\Gamma(\alpha+2n)}\;_2F_1\left(\begin{matrix} -n,\,-n+\frac{1}{2}\\[.5ex]
-\alpha-2n+\frac{3}{2}\end{matrix}\,;\,1\right)
\\
&=\frac{2\sqrt{\pi}}{(v+1)^{\alpha+n-\frac{1}{2}}}\frac{\Gamma\left(\alpha+2n -\frac{1}{2}\right)}{\Gamma(\alpha+2n)} \frac{\Gamma\left(-\alpha-2n+\frac{3}{2}\right)\Gamma\left(-\alpha+1\right)}{\Gamma\left(-\alpha- n+1\right)\Gamma\left(-\alpha- n+\frac{3}{2}\right)}
\\
&=\frac{2\sqrt{\pi}}{(v+1)^{\alpha+n-\frac{1}{2}}}\frac{\Gamma(\alpha+n)\Gamma\left(\alpha+n -\frac{1}{2}\right)}{\Gamma(\alpha)\Gamma(\alpha+2n)}.
\end{align*}
It follows that
\begin{gather*}
Q'(v)=-\frac{2\sqrt{\pi}}{(v+1)^{\alpha+n+\frac{1}{2}}}\frac{\Gamma(\alpha+n)\Gamma
\left(\alpha+n +\frac{1}{2}\right)}{\Gamma(\alpha)\Gamma(\alpha+2n)}.
\end{gather*}
Next, we compute
\begin{align*}
J(x)&= \int_{-\infty}^{\infty} \frac{1}{\big(x+1+t^2 \big)^{\alpha+n+\frac{1}{2}}}\big(\sqrt{x+1+t^2}-t\big)^{2n}{\rm d}t
\\
&=2\int_0^{\infty}\frac{1}{\big(x+1+t^2 \big)^{\alpha+n+\frac{1}{2}}}\sum_{r=0}^n\begin{pmatrix} 2n\\2r\end{pmatrix} (x+1+t^2)^{n-r} t^{2r}{\rm d}t
\\
&=2\sum_{r=0}^n\begin{pmatrix} 2n\\2r\end{pmatrix}\int_0^{\infty}\frac{t^{2r}}{\big(x+1+t^2 \big)^{\alpha+r+\frac{1}{2}}} {\rm d}t
\\
&= \frac{1}{(x+1)^{\alpha}}\sum_{r=0}^n\frac{(-n)_r\left(-n+\frac{1}{2}\right)_r} {\left(\frac{1}{2}\right)_rr!}\frac{\Gamma\left(r+\frac{1}{2}\right)\Gamma\left(\alpha\right)}
{\Gamma\left(\alpha+r+\frac{1}{2}\right)}.
\end{align*}
It follows that
\begin{align*}
J(x)&= \frac{\Gamma(\alpha)}{\Gamma\left(\alpha+\frac{1}{2}\right)} \frac{\sqrt{\pi}}{(x+1)^{\alpha}}\sum_{r=0}^n\frac{(-n)_r\left(-n+\frac{1}{2}\right)_r}{ r!}\frac{ 1}
{ \left(\alpha +\frac{1}{2}\right)_r}
\\
& = \frac{\Gamma(\alpha)}{\Gamma\left(\alpha+\frac{1}{2}\right)} \frac{\sqrt{\pi}}{(x+1)^{\alpha}}\;_2F_1\left(\begin{matrix} -n,\,-n+\frac{1}{2}
\\
\alpha +\frac{1}{2}\end{matrix}\,;\,1\right)
\\
&= \frac{\Gamma(\alpha)}{\Gamma\left(\alpha+\frac{1}{2}\right)} \frac{\sqrt{\pi}}{(x+1)^{\alpha}}\frac{\Gamma\left(\alpha+\frac{1}{2}\right)\Gamma(\alpha+2n)}
{\Gamma(\alpha+n)\Gamma\left(\alpha+n+\frac{1}{2}\right)}.
\end{align*}
This proves that for
\begin{gather*}
\Psi(x)=\frac{1}{(x+1)^{\alpha}},
\end{gather*}
the inversion formula holds. Namely, if
\begin{gather*}
Q(v)=2(-1)^n\int_{-\infty}^{\infty} \Psi\big(x^2+v\big)\frac{1}{\big(x-{\rm i}\sqrt{v+1}\big)^{2n}} {\rm d}x,
\end{gather*}
then
\begin{gather*}
 \Psi(x) =-\frac{1}{2\pi}\int_{-\infty}^{\infty} Q'\big(x+t^2\big) \big(\sqrt{x+1+t^2}-t\big)^{2n}{\rm d}t.
 \end{gather*}
 Thus, this formula works for $\Psi(u)=\Psi_{n,s}(u)-\Psi_{n,a}(u)$.

 \section{The hyperbolic contribution}\label{a1}
 In this section, we want to compute the hyperbolic contribution
\begin{gather*}
\Xi_H=\sum_{\substack{\gamma\in\Gamma\\\gamma \;\text{is hyperbolic}}}\iint_{\Gamma\backslash\mathbb{H}}k(\gamma z,z)\gamma'(z)^ny^{2n}{\rm d}\mu(z)
\end{gather*}
to the trace equation~\eqref{eq0903_6}.

For any $\gamma\in \Gamma$, let
\begin{gather*}
\Gamma_{\gamma}=\big\{\alpha\in\Gamma\mid \alpha\gamma\alpha^{-1}=\gamma\big\}.
\end{gather*}
It is a cyclic subgroup of $\Gamma$ generated by some $\gamma_0\in\Gamma$ such that $\gamma=\gamma_0^{\ell}$ for some nonzero integer~$\ell$.

The elements in $\Gamma$ are partitioned into conjugacy classes. The conjugacy class of $\gamma$ consists of all elements in $\Gamma$ of the form $\alpha\gamma\alpha^{-1}$, when $\alpha$ runs through all elements in $\Gamma$. Two ele\-ments~$\alpha$ and~$\beta$ in $\Gamma$ define the same element in the class containing $\gamma$, namely,
 \begin{gather*}
 \alpha\gamma\alpha^{-1}=\beta\gamma\beta^{-1},
 \end{gather*}
 if and only if $\alpha^{-1}\beta\in \Gamma_{\gamma}$.

Let $\gamma\in\Gamma$ be a hyperbolic element.
There exists $\sigma\in {\rm PSL}(2,\mathbb{R})$ and a real number $\lambda_{\gamma}>1$ such that
 \begin{gather*}
 \sigma^{-1} \gamma\sigma=D_{\gamma}=
 \begin{pmatrix} \lambda_{\gamma}^{1/2} & 0 \\0 &\lambda_{\gamma}^{-1/2}\end{pmatrix}\!.
 \end{gather*}
 $\lambda_{\gamma}$ is called the multiplier of $\gamma$.
Each hyperbolic element $\gamma$ in $\Gamma$ belongs to a conjugacy class. All elements in the same conjugacy class has the same multiplier. There exists $\gamma_0\in \Gamma$ such that $\gamma=\gamma_0^{\ell}$ for some nonzero integer $\ell$, and $\Gamma_{\gamma}$ is generated by $\gamma_0$. $\gamma_0$ is called a primitive hyperbolic element.

Let $P$ be the set of conjugacy classes of primitive hyperbolic elements in $\Gamma$. Then the set of hyperbolic elements in $\Gamma$ can be written as
\begin{gather*}
\bigcup_{[\gamma_0]\in P}\bigcup_{\ell=1}^{\infty}\big\{ \alpha\gamma_0^{\ell}\alpha^{-1}\mid \alpha\in \Gamma_{\gamma_0}\backslash\Gamma \big\}.
\end{gather*}

 Given $\gamma_0$ a representative of a primmitive hyperbolic conjugacy class, and $\ell$ a positive integer, let $\gamma=\gamma_0^{\ell}$. We first consider the integral
\begin{gather*}
I_H= \sum_{\alpha\in \Gamma_{\gamma}\backslash\Gamma}\iint_Fk\big(\alpha\gamma\alpha^{-1} z,z\big)\left[(\alpha\gamma\alpha)'(z)\right]^ny^{2n}{\rm d}\mu(z).\end{gather*}Using the invariant property of the kernel $k(z,w)$, we find that
\begin{gather*}
I_H
= \iint_{ \sigma^{-1}\left(\Gamma_{\gamma}\backslash\mathbb{H}\right)}k( \lambda_{\gamma} z, z)\lambda_{\gamma}^n y^{2n} {\rm d}\mu(z).
\end{gather*}
In the following, we abbreviate $\lambda_{\gamma}$ as $\lambda$.
 Let $\lambda=\lambda_0^{\ell}$. Then
\begin{gather*}\sigma^{-1}\left(\Gamma_{\gamma}\backslash\mathbb{H}\right)=\left\{x+{\rm i}y\mid 1\leq y\leq \lambda_0\right\}.\end{gather*}
It follows that
\begin{align*}
I_H&=\lambda^n\int_1^{\lambda_0}\int_{-\infty}^{\infty}k( \lambda z, z)y^{2n}{\rm d}x \frac{{\rm d}y}{y^2}
\\
&= \lambda^n\int_1^{\lambda_0}\int_{-\infty}^{\infty}\Psi \bigg(\frac{(\lambda-1)^2}{4\lambda}\big(x^2+1\big)\bigg)\frac{(-4)^n}{[(\lambda-1)x+{\rm i}(\lambda+1)]^{2n}}{\rm d}x \frac{{\rm d}y}{y}
\\
&=\frac{2(-1)^n\lambda^{\frac{1}{2}}}{\lambda-1}\ln\lambda_0\int_{-\infty}^{\infty}
\Psi\bigg(x^2+\frac{(\lambda-1)^2}{4\lambda} \bigg)\frac{1}{\big[x+{\rm i}\tfrac{\lambda+1}{2\sqrt{\lambda}}\big]^{2n}}{\rm d}x
\\
&=\frac{\sqrt{\lambda}}{\lambda-1}\ln\lambda_0 \; g(\ln\lambda)
=\frac{1}{2s+2n-1}\frac{\lambda^{-s-n}}{1-\lambda^{-1}}\ln\lambda_0 -\frac{1}{2a+2n-1}\frac{\lambda^{-a-n}}{1-\lambda^{-1}}\ln\lambda_0.
\end{align*}
Now we can sum up the contributions from $\gamma=\gamma_0^{\ell}$, where $\ell$ ranges from 1 to $\infty$. To avoid confusion with the eigenvalues of Laplacian, we denote $\lambda_{\gamma}$ by $N(\gamma)$.
Since
\begin{gather*}
\sum_{\ell=1}^{\infty}\frac{\lambda^{-\ell(s+n)}}{1-\lambda^{-\ell}}
=\sum_{k=0}^{\infty}\sum_{\ell=1}^{\infty}\lambda^{-\ell(s+n+k)}
=\sum_{k=0}^{\infty}\frac{1}{\lambda^{s+k+n}-1},
\end{gather*}
we find that $\Xi_H=\mathscr{E}_H(s)-\mathscr{E}_H(a)$, where
\begin{gather*}
\mathscr{E}_H(s)= \frac{1}{2s+2n-1}\sum_{[\gamma]\in P}\sum_{k=0}^{\infty}\frac{N(\gamma)^{-(s+n+k)}}{1-N(\gamma)^{-(s+n+k)}}\ln N(\gamma).
\end{gather*}

\section{Elliptic contribution}\label{a2}
 In this appendix, we compute the elliptic contribution
 \begin{gather*}
\Xi_E=\sum_{\substack{\gamma\in\Gamma\\\gamma \;\text{is elliptic}}}\iint_{\Gamma\backslash\mathbb{H}}k(\gamma z,z)\gamma'(z)^ny^{2n}{\rm d}\mu(z)
\end{gather*}to the trace equation~\eqref{eq0903_6}.

For simplicity, let us just do the calculation with \begin{gather*}\Psi(u)=\Psi_{n,s}(u),\end{gather*} understanding that the singularity at $u=0$ is eliminated by subtracting $\Psi_{n,a}(u)$. At the end, we will restore the result with \begin{gather*}\Psi(u)=\Psi_{n,s}(u)-\Psi_{n,a}(u).\end{gather*}
Recall that the Riemann surface $X$ has $v$ ramification points corresponding to the $v$ elliptic elements $\tau_1, \dots, \tau_v$ of orders $m_1, \dots, m_v$ respectively. For each $1\leq j\leq v$, there is a
$\sigma_j\in{\rm PSL}(2,\mathbb{R})$ such that
\begin{gather*}
\sigma_j^{-1}\tau_j\sigma_j=
\begin{pmatrix} \cos\theta_j & -\sin\theta_j\\ \sin\theta_j & \cos\theta_j\end{pmatrix}\!,
\end{gather*}
where $\theta_j=\frac{\pi}{m_j}$. If $\gamma$ is an elliptic element of $\Gamma$, then there exists $\alpha\in \Gamma$, $1\leq j\leq v$ and an integer $1\leq \ell\leq m_j-1$ so that
\begin{gather*}
\alpha^{-1}\gamma\alpha=\tau_j^{\ell}.
\end{gather*}
This implies that
\begin{align*}
\Xi_E&=\sum_{j=1}^v\sum_{\ell=1}^{m_j-1}\sum_{ \alpha\in\Gamma_{\tau_j}\backslash\Gamma} \iint_{\Gamma\backslash\mathbb{H}}k\big(\alpha \tau_j^{\ell}\alpha^{-1} z,z\big)\big[\big(\alpha\tau_j^{\ell}\alpha^{-1}\big)'(z)\big]^ny^{2n}{\rm d}\mu(z)
\\
 &=\sum_{j=1}^v\sum_{\ell=1}^{m_j-1} \iint_{\Gamma_{\tau_j}\backslash\mathbb{H}}k\big( \tau_j^{\ell} z,z\big)\big(\tau_j^{\ell}\big)'(z)^ny^{2n}{\rm d}\mu(z)
 =\sum_{j=1}^v\sum_{\ell=1}^{m_j-1} I_{E, j}(\theta_{j, \ell}),
 \end{align*}
 where $\theta_{j,\ell}= \frac{\pi \ell}{m_j}$, and
\begin{gather*}
 I_{E}(\theta)= \frac{1}{m_j}\iint_{ \mathbb{H}}k(R_{\theta} z,z)R_{\theta}'(z)^ny^{2n}{\rm d}\mu(z),
 \\[.5ex]
 R_{\theta}= \begin{pmatrix} \cos\theta & -\sin\theta \\ \sin\theta & \cos\theta \end{pmatrix}\!.
\end{gather*}
In the following, we omit the subscript $j$ in our calculation of $I_{E,j}$.
Now,
\begin{gather*}
u(R_{\theta}z,z)+1=\frac{\sin^2\theta\big(1+x^2+y^2\big)^2+4\cos^2\theta y^2}{4y^2},
\\
(-4)^n\frac{R_{\theta}'(z)^n}{(R_{\theta}z-\bar{z})^{2n}}
=\frac{(-4)^n}{\big(\sin\theta \big(1+x^2+y^2\big)-2{\rm i}y\cos\theta\big)^{2n}}.
\end{gather*}
With
\begin{gather*}
t=\frac{1+x^2+y^2}{2y},
\end{gather*}
we have
\begin{gather*}
x=\sqrt{2yt-1-y^2},\qquad
{\rm d}x=\frac{y\,{\rm d}t}{\sqrt{2yt-1-y^2}}.
\end{gather*}
Then
\begin{align*}
I_{E}(\theta)&=(-1)^n\frac{2}{m}\int_0^{\infty}\int_{\frac{1+y^2}{2y}}^{\infty}\frac{\Psi \big( t^2\sin^2\theta -\sin^2\theta \big)}{(t\sin\theta-{\rm i}\cos\theta)^{2n}}\frac{{\rm d}t}{\sqrt{2yt-1-y^2}}\frac{{\rm d}y}{y}
\\
&=(-1)^n\frac{2}{m}\int_1^{\infty}\int_{t-\sqrt{t^2-1}}^{t+\sqrt{t^2-1}}\frac{{\rm d}y}{y\sqrt{2yt-1-y^2}}\frac{\Psi \big(t^2\sin^2\theta -\sin^2\theta\big)}{(t\sin\theta-{\rm i}\cos\theta)^{2n}}{\rm d}t.
\end{align*}It has been computed that
\begin{gather*}
\int_{t-\sqrt{t^2-1}}^{t+\sqrt{t^2-1}}\frac{{\rm d}y}{y\sqrt{2yt-1-y^2}}= \pi.
\end{gather*}
Therefore,
\begin{gather*}
I_{E}(\theta)
= (-1)^n\frac{2\pi}{m}\int_1^{\infty}
\frac{\Psi\big(t^2\sin^2\theta -\sin^2\theta\big)}{\big(t\sin\theta-{\rm i}\cos\theta\big)^{2n}}{\rm d}t.
\end{gather*}
This integral is difficult to compute. Fischer \cite{Fischer1987} used the power series expansion of $\Psi_{n,s}(u)$ followed by a lot of tedious calculations. The generalization from $n=0$ to general positive $n$ is not easy, and Hejhal was unable to proceed in his first volume~\cite{Hejhal1976}; only succeeded to circumvent the problem in his second volume~\cite{Hejhal1983}.
 We will use the method of Hejhal~\cite{Hejhal1983}.

 First of all, using the substitution $u=t^2\sin^2\theta-\sin^2\theta$, we find that
 \begin{gather*}
I_{E}(\theta)
= (-1)^n\frac{ \pi}{m\sin\theta}\int_0^{\infty} \frac{\Psi(u)}{\big(\sqrt{u+\sin^2\theta}-{\rm i}\cos\theta\big)^{2n}} \frac{{\rm d}u}{\sqrt{u+\sin^2\theta}}.
\end{gather*}
Using the inversion formula (see Appendix~\ref{ainversion})
\begin{gather*}
\Psi (u)= -\frac{1}{2\pi}\int_{-\infty}^{\infty}Q'\big(u+t^2\big)\big(\sqrt{u+1+t^2}-t\big)^{2n} {\rm d}t,
\end{gather*}
we find that
\begin{align*}
I_{E}(\theta)
&=\frac{(-1)^{n+1}}{2m\sin\theta}\int_0^{\infty}\int_{-\infty}^{\infty} Q'\big(u+t^2\big)\frac{\big(\sqrt{u+1+t^2}-t\big)^{2n}}{\big(\sqrt{u+\sin^2\theta}-{\rm i}\cos\theta\big)^{2n}} {\rm d}t\frac{{\rm d}u}{\sqrt{u+\sin^2\theta}}
\\
&=\frac{(-1)^{n+1}}{2m\sin\theta}\int_{-\infty}^{\infty}\int_{t^2}^{\infty} Q'(v)\frac{\big(\sqrt{v+1}-t\big)^{2n}}{\big(\sqrt{v-t^2+\sin^2\theta}-{\rm i}\cos\theta\big)^{2n}} \frac{{\rm d}v}{\sqrt{v-t^2+\sin^2\theta}}{\rm d}t
\\
&=\frac{(-1)^{n+1}}{2m\sin\theta}\int_{0}^{\infty}Q'(v)\mathscr{H}(v)\,{\rm d}v,
\end{align*}
where
\begin{gather*}
\mathscr{H}(v)=\int_{-\sqrt{v}}^{\sqrt{v}} \frac{\big(\sqrt{v+1}-t\big)^{2n}}{\big(\sqrt{v-t^2+\sin^2\theta}-{\rm i}\cos\theta\big)^{2n}} \frac{{\rm d}t}{\sqrt{v-t^2+\sin^2\theta}}.
\end{gather*}
Making a change of variables
\begin{gather*}
t=\sqrt{v+\sin^2\theta}\sin\varphi,
\end{gather*}
we find that
\begin{gather*}
\mathscr{H}(v)= \int_{-\varphi_0(v)}^{\varphi_0(v)} \frac{\big(\sqrt{v +1}-\sqrt{v+\sin^2\theta}\sin\varphi\big)^{2n} }{\big(\sqrt{v +\sin^2\theta}\cos\varphi-{\rm i}\cos\theta\big)^{2n}}{\rm d}\varphi,
\end{gather*}
where
\begin{gather*}
\varphi_0(v)=\sin^{-1}\frac{\sqrt{v}}{\sqrt{v+\sin^2\theta}}.
\end{gather*}
By definition, $\mathscr{H}(0)=0$, and hence,
\begin{gather*}
I_{E}(\theta)= \frac{(-1)^{n}}{2m\sin\theta}\int_{0}^{\infty}Q(v)\mathscr{H}'(v)\,{\rm d}v.
\end{gather*}
Now, we calculate $\mathscr{H}'(v)$. It is straightforward to find that
\begin{gather*}
\mathscr{H}'(v)=\varphi_0'(v)\left[\frac{\big(\sqrt{v +1}-\sqrt{v+\sin^2\theta}\sin\varphi_0\big)^{2n} }{\big(\sqrt{v +\sin^2\theta}\cos\varphi_0-{\rm i}\cos\theta\big)^{2n}}
+\frac{\big(\sqrt{v +1}+\sqrt{v+\sin^2\theta}\sin\varphi_0\big)^{2n} }{\big(\sqrt{v +\sin^2\theta}\cos\varphi_0-{\rm i}\cos\theta\big)^{2n}}\right]
\\ \hphantom{\mathscr{H}'(v)=}
{}+ \int_{-\varphi_0(v)}^{\varphi_0(v)} \frac{\rm d}{{\rm d}v}\left[\frac{ \sqrt{v +1}-\sqrt{v+\sin^2\theta}\sin\varphi }{ \sqrt{v +\sin^2\theta}\cos\varphi-{\rm i}\cos\theta }\right]^{2n}{\rm d}\varphi.
\end{gather*}
If we define $\omega$ so that
\begin{gather*}
\sinh\omega=\frac{\cos\theta}{\sqrt{v+\sin^2\theta}},
\end{gather*}
we find that
\begin{gather*}
\frac{\sqrt{v+1}}{\sqrt{v+\sin^2\theta}}=\cosh\omega.
\end{gather*}
Observe that
\begin{align*}
\frac{\rm d}{{\rm d}v} \frac{ \sqrt{v +1}-\sqrt{v+\sin^2\theta}\sin\varphi }{ \sqrt{v +\sin^2\theta}\cos\varphi-{\rm i}\cos\theta } &=\frac{\rm d}{{\rm d}\omega}
\frac{\cosh\omega-\sin\varphi}{\cos\varphi-{\rm i}\sinh\omega} \frac{{\rm d}\omega}{{\rm d}v}
=-{\rm i}\frac{\rm d}{{\rm d}\varphi}
\frac{\cosh\omega-\sin\varphi}{\cos\varphi-{\rm i}\sinh\omega} \frac{{\rm d}\omega}{{\rm d}v}
\\
&=-\frac{{\rm i}\cos\theta}{2(v+\sin^2\theta)\sqrt{v+1}}\frac{\rm d}{{\rm d}\varphi}
\frac{\cosh\omega-\sin\varphi}{\cos\varphi-{\rm i}\sinh\omega} .
\end{align*}
These imply that
\begin{gather*}
\mathscr{H}'(v)=\frac{\sin\theta}{2\sqrt{v}(v+\sin^2\theta)}
\left[\frac{\big(\sqrt{v +1}-\sqrt{v}\big)^{2n} }{(\sin\theta-{\rm i}\cos\theta)^{2n}}
+\frac{\big(\sqrt{v +1}+\sqrt{v}\big)^{2n} }{(\sin\theta-{\rm i}\cos\theta)^{2n}}\right]
\\ \hphantom{\mathscr{H}'(v)=}
{}-\frac{{\rm i}\cos\theta}{2(v+\sin^2\theta)\sqrt{v+1}} \int_{-\varphi_0(v)}^{\varphi_0(v)} \frac{\rm d}{{\rm d}\varphi}\left[\frac{\sqrt{v +1}-\sqrt{v+\sin^2\theta}\sin\varphi }{ \sqrt{v +\sin^2\theta}\cos\varphi-{\rm i}\cos\theta}\right]^{2n}{\rm d}\varphi
\\ \hphantom{\mathscr{H}'(v)}
{}=\frac{(-1)^n\sin\theta\, {\rm e}^{2{\rm i}n\theta}}{2\sqrt{v}(v+\sin^2\theta)}\big[\big(\sqrt{v +1}-\sqrt{v}\big)^{2n}+ \big(\sqrt{v +1}+\sqrt{v}\big)^{2n}\big]
\\ \hphantom{\mathscr{H}'(v)=}
{}-\frac{(-1)^n{\rm i}\cos\theta\, {\rm e}^{2{\rm i}n\theta}}{2(v+\sin^2\theta)\sqrt{v+1}}
\big[\big(\sqrt{v +1}-\sqrt{v}\big)^{2n}-\big(\sqrt{v +1}+\sqrt{v}\big)^{2n}\big].
\end{gather*}

It follows that
\begin{gather*}
\mathscr{H}'\bigg(\!\sinh^2\frac{t}{2}\bigg)\frac{\rm d}{{\rm d}t}\sinh^2\frac{t}{2}
\\ \qquad
{}=\frac{(-1)^n{\rm e}^{2{\rm i}n\theta}}{2\big(\sinh^2\frac{t}{2}+\sin^2\theta\big)}\bigg\{\sin\theta \cosh\frac{t}{2}({\rm e}^{tn}+{\rm e}^{-tn})+{\rm i}\cos\theta\sinh\frac{t}{2}({\rm e}^{tn}-{\rm e}^{-tn})\bigg\}
\\ \qquad
{}=\frac{(-1)^n{\rm e}^{2{\rm i}n\theta}}{\cosh t-\cos 2\theta}\Bigg\{\sin\theta \frac{\big({\rm e}^{\frac{t}{2}}+{\rm e}^{-\frac{t}{2}}\big)({\rm e}^{tn}+{\rm e}^{-tn})}{2}+{\rm i}\cos\theta\frac{\big({\rm e}^{\frac{t}{2}}-{\rm e}^{-\frac{t}{2}}\big)({\rm e}^{tn}-{\rm e}^{-tn})}{2}\Bigg\}
\\ \qquad
{}=\frac{(-1)^n{\rm e}^{2{\rm i}n\theta}}{\cosh t-\cos 2\theta}\Bigg\{\sin\theta \frac{ {\rm e}^{\left(n+\frac{1}{2}\right)t} +{\rm e}^{-\left(n+\frac{1}{2}\right)t}+{\rm e}^{\left(n-\frac{1}{2}\right)t}+{\rm e}^{-\left(n-\frac{1}{2}\right)t} }{2}
\\ \qquad\hphantom{=\frac{(-1)^n{\rm e}^{2{\rm i}n\theta}}{\cosh t-\cos 2\theta}\bigg\{}
{} +{\rm i}\cos\theta\frac{ {\rm e}^{\left(n+\frac{1}{2}\right)t} +{\rm e}^{-\left(n+\frac{1}{2}\right)t}-{\rm e}^{\left(n-\frac{1}{2}\right)t}-{\rm e}^{-\left(n-\frac{1}{2}\right)t} }{2}\Bigg\}.
\end{gather*}
This implies that
\begin{gather*}
I_{E}(\theta)
=\frac{(-1)^n}{2m\sin\theta}\int_{0}^{\infty}Q\bigg(\sinh^2\frac{t}{2}\bigg)H'
\bigg(\sinh^2\frac{t}{2}\bigg)\frac{\rm d}{{\rm d}t}\sinh^2\frac{t}{2}{\rm d}t
\\ \hphantom{I_{E}(\theta)}
{}=\frac{{\rm e}^{2{\rm i}n\theta}}{2m\sin\theta}\int_{0}^{\infty}\frac{g(t)}{ \cosh t -\cos 2\theta}\Bigg\{\sin\theta \frac{ {\rm e}^{\left(n+\frac{1}{2}\right)t} +{\rm e}^{-\left(n+\frac{1}{2}\right)t}+{\rm e}^{\left(n-\frac{1}{2}\right)t}+{\rm e}^{-\left(n-\frac{1}{2}\right)t} }{2}
\\ \hphantom{I_{E}(\theta)=\frac{{\rm e}^{2{\rm i}n\theta}}{2m\sin\theta}\int_{0}^{\infty}}
{} +{\rm i}\cos\theta\frac{ {\rm e}^{\left(n+\frac{1}{2}\right)t} +{\rm e}^{-\left(n+\frac{1}{2}\right)t}-{\rm e}^{\left(n-\frac{1}{2}\right)t}-{\rm e}^{-\left(n-\frac{1}{2}\right)t} }{2}\Bigg\}.
\end{gather*}
Using
\begin{gather*}
g(t)=\frac{1}{2s+2n-1}{\rm e}^{-t\left(s+n-\frac{1}{2}\right)},
\end{gather*}
and
\begin{gather*}
\int_0^{\infty}\frac{{\rm e}^{-\mu t}}{\cosh t-\cos 2\theta}{\rm d}t=\frac{2}{\sin 2\theta}\sum_{k=1}^{\infty} \frac{\sin 2k\theta}{\mu+k},
\end{gather*}
we find that
\begin{gather*}
I_{E}(\theta)=\frac{1}{(2s+2n-1)}\frac{ {\rm e}^{2{\rm i}n\theta}}{2m\sin\theta}\frac{1}{\sin 2\theta}
\\ \hphantom{I_{E}(\theta)=}
{}\times\Bigg\{\sin\theta \bigg[ \sum_{k=1}^{\infty} \frac{\sin 2k\theta}{s+k-1}+\sum_{k=1}^{\infty} \frac{\sin 2k\theta}{s+k}+\sum_{k=1}^{\infty} \frac{\sin 2k\theta}{s+k+2n}+\sum_{k=1}^{\infty} \frac{\sin 2k\theta}{s+k+2n-1}\bigg]
\\ \hphantom{I_{E}(\theta)=\times\Bigg\{}
{}+{\rm i}\cos\theta \bigg[\sum_{k=1}^{\infty} \frac{\sin 2k\theta}{s+k-1}-\!\sum_{k=1}^{\infty} \frac{\sin 2k\theta}{s+k}+\!\sum_{k=1}^{\infty} \frac{\sin 2k\theta}{s+k+2n}-\!\sum_{k=1}^{\infty} \frac{\sin 2k\theta}{s+k+2n-1}\bigg]\!\Bigg\}
\\ \hphantom{I_{E}(\theta)}
{}=\frac{1}{(2s+2n-1)}\frac{ {\rm e}^{2{\rm i}n\theta}}{2m\sin\theta}\frac{1}{\sin 2\theta}
\\ \hphantom{I_{E}(\theta)=}
{}\times\Bigg\{\sin\theta \bigg[ \sum_{k=0}^{\infty} \frac{\sin (2k+2)\theta}{s+k}+\sum_{k=0}^{\infty} \frac{\sin 2k\theta}{s+k}+\sum_{k=0}^{\infty} \frac{\sin 2k\theta}{s+k+2n}+\sum_{k=0}^{\infty} \frac{\sin (2k+2)\theta}{s+k+2n}\bigg]
\\ \hphantom{I_{E}(\theta)=\times\Bigg\{}
{}+{\rm i}\cos\theta \bigg[ \sum_{k=0}^{\infty} \frac{\sin (2k\!+\!2)\theta}{s+k}-\!\sum_{k=0}^{\infty} \frac{\sin 2k\theta}{s+k}+\!\sum_{k=0}^{\infty} \frac{\sin 2k\theta}{s+k+2n}-\!\sum_{k=0}^{\infty} \frac{\sin (2k\!+\!2)\theta}{s+k+2n}\bigg]\!\Bigg\}
\\ \hphantom{I_{E}(\theta)}
{}=\frac{1}{2s+2n-1}\frac{ i{\rm e}^{2{\rm i}n\theta}}{2m\sin\theta}\Bigg(\sum_{k=0}^{\infty}\frac{{\rm e}^{-{\rm i}(2k+1)\theta}}{s+k}-\sum_{k=0}^{\infty}\frac{{\rm e}^{{\rm i}(2k+1)\theta}}{s+2n+k}\Bigg).
\end{gather*}

Now we find the sum
\begin{gather*}
\sum_{\ell=1}^{m-1}I_{E}\bigg(\frac{\pi\ell}{m}\bigg).
\end{gather*}
Notice that is $k$ is an integer,
\begin{gather*}
S_k=\sum_{\ell=1}^{m-1}\frac{\rm i} {\sin \frac{\pi \ell}{m}}\exp\bigg({-}\frac{\pi{\rm i} \ell(2k+1)}{m}\bigg)
\\ \hphantom{S_k}
{}=\frac{1}{2}\Bigg[\sum_{\ell=1}^{m-1}\frac{\rm i} {\sin \frac{\pi \ell}{m}}
\exp\bigg({-}\frac{\pi{\rm i} \ell(2k+1)}{m}\bigg)
+\sum_{\ell=1}^{m-1}\frac{\rm i}{\sin \frac{\pi (m-\ell)}{m}}
\exp\bigg({-}\frac{\pi{\rm i} (m-\ell)(2k+1)}{m}\bigg)\Bigg]
\\ \hphantom{S_k}
{}=\sum_{\ell=1}^{m-1}\frac{\exp\left(\frac{\pi{\rm i} \ell(2k+1)}{m}\right)-\exp\left(-\frac{\pi{\rm i} \ell(2k+1)}{m}\right)}{\exp \left(\frac{\pi{\rm i} \ell}{m}\right)-\exp\left(- \frac{\pi{\rm i} \ell}{m}\right)}
=\sum_{\ell=1}^{m-1}\frac{x_{\ell}^{2k+1}-x_{\ell}^{-2k-1}}{x_{\ell}-x_{\ell}^{-1}},
\end{gather*}
where
\begin{gather*}
x_{\ell}=\exp\bigg(\frac{\pi{\rm i} \ell}{m}\bigg).
\end{gather*}
Using the fact that
\begin{gather*}
\sum_{\ell=1}^{m-1}x_{\ell}^{2h}
=\begin{cases}
m-1 &\text{if} \ m\text{ divides }h,
\\
-1 & \text{if} \ m\text{ does not divide }h,
\end{cases}
\end{gather*}
we find that if $k$ is a nonnegative integer
\begin{align*}
S_k&=\sum_{\ell=1}^{m-1}\big(x_{\ell}^{2k}+x_{\ell}^{2k-2}+\dots+x_{\ell}^{-2k+2}+x_{\ell}^{-2k}\big)
=-2k-1+m\bigg(2\bigg\lfloor\frac{k}{m}\bigg\rfloor+1\bigg)
\\
&=-2m\bigg\{\frac{k}{m}\bigg\}+m-1.
\end{align*}
Here $\{x\}=x-\lfloor x\rfloor$.
The case when $k$ is a negative integer is more complicated. First we notice that for any integer $x$,
if $q=\left\lfloor \frac{x}{m}\right\rfloor$, then $x=qm+r$, where $0\leq r\leq m-1$. It follows that
\begin{gather*}
-x-1=-qm-r-1=(-q-1)m+m-r-1,
\end{gather*}
where
\begin{gather*}
0\leq m-r-1\leq m-1.
\end{gather*}
This shows that
\begin{gather*}
\bigg\lfloor\frac{-x-1}{m}\bigg\rfloor=-q-1=-\bigg\lfloor\frac{x}{m}\bigg\rfloor-1.
\end{gather*}

If $k$ is a negative integer, let $k'=-k$. Then $k'$ is a positive integer and
\begin{align*}
S_k&=-\sum_{\ell=1}^{m-1}\frac{x_{\ell}^{2k'-1}-x_{\ell}^{-2k'+1}}{x_{\ell}-x_{\ell}^{-1}}
\\
&=-\sum_{\ell=1}^{m-1}\big(x_{\ell}^{2k'-2}+x_{\ell}^{2k'-4}+\dots+x_{\ell}^{-2k'+4}+x_{\ell}^{-2k'+2}\big)
\\
&= -\bigg\{{-}2k'+1+m\bigg(2\bigg\lfloor\frac{k'-1}{m}\bigg\rfloor+1\bigg)\bigg\}
=-2k-1-m\bigg({-}2\bigg\lfloor\frac{k}{m}\bigg\rfloor-1\bigg)
\\
&=-2m\bigg\{\frac{k}{m}\bigg\}+m-1.
\end{align*}
This shows that for any integer $k$,
\begin{gather*}
S_k=-2\alpha_m(k)+m-1,
\end{gather*}
where
\begin{gather*}
\alpha_m(k)=m\bigg\{\frac{k}{m}\bigg\}
\end{gather*}
is the least positive residue of $k$ modulo $m$. We have established that
\begin{gather*}
\alpha_m(-k)=m-1-\alpha_m(k).
\end{gather*}
Hence,
\begin{gather*}S_{-k}=-S_k.
\end{gather*}
 It follows that
\begin{align*}
\sum_{\ell=1}^{m-1}I_{E}\bigg(\frac{\pi\ell}{m}\bigg)
={}&\frac{1}{2m}\frac{1}{2s+2n-1}\Bigg\{\sum_{k=0}^{\infty}\frac{1}{s+k}(-2\alpha_m(k-n)+m-1)
\\
&+\sum_{k=0}^{\infty}\frac{1}{s+2n+k}(-2\alpha_m(k+n)+m-1)\Bigg\}.
\end{align*}

Writing $k=mq+r$, where $0\leq r\leq m-1$, we have
\begin{align*}
\sum_{\ell=1}^{m-1}I_{E}\bigg(\frac{\pi\ell}{m}\bigg)
={}&\frac{1}{2m}\frac{1}{2s+2n-1}\Bigg\{\sum_{q=0}^{\infty}\sum_{r=0}^{m-1}\frac{1}{s +mq+r}(-2\alpha_m(r-n)+m-1)
\\
& +\sum_{q=0}^{\infty}\sum_{r=0}^{m-1}\frac{1}{s+2n+mq+r}(-2\alpha_m(r+n)+m-1)\Bigg\}.
\end{align*}

Notice that when $r$ runs from $0$ to $m-1$, $r-n$ and $r+n$ respectively runs through a complete residue system modulo $m$. Hence,
\begin{gather*}\sum_{r=0}^{m-1} \alpha_{m}(r-n)
=\sum_{r=0}^{m-1}[m-1-\alpha_{m}(r-n)]
\end{gather*}
and
\begin{gather*}
\sum_{r=0}^{m-1} \alpha_{m}(r+n)
=\sum_{r=0}^{m-1}[m-1-\alpha_{m}(r+n)].
\end{gather*}
Therefore,
\begin{align}\label{eq314_1}
\sum_{r=0}^{m-1}(-2\alpha_m(r-n)+m-1)=\sum_{r=0}^{m-1}(-2\alpha_m(r+n)+m-1)=0.
\end{align}
Thus
\begin{gather*}
\sum_{\ell=1}^{m-1}\!I_{E}\bigg(\frac{\pi\ell}{m}\bigg) =\frac{1}{2m}\frac{1}{2s\!+2n\!-1}\Bigg\{\sum_{q=0}^{\infty}\sum_{r=0}^{m-1}
\bigg[\frac{1}{s\!+ mq\!+r}-\frac{1}{m(q+1)}\bigg](-2\alpha_m(r-n)+m-1)
\\ \hphantom{\sum_{\ell=1}^{m-1}I_{E}\bigg(\frac{\pi\ell}{m}\bigg) =}
{} +\sum_{q=0}^{\infty}\sum_{r=0}^{m-1}\bigg[\frac{1}{s+2n+mq+r}-\frac{1}{m(q+1)}\bigg] (-2\alpha_m(r+n)+m-1)\Bigg\}
\\ \hphantom{\sum_{\ell=1}^{m-1}I_{E}\bigg(\frac{\pi\ell}{m}\bigg)}
{}=\frac{1}{2s+2n-1}\sum_{r=0}^{m-1}\bigg[\frac{2\alpha_m(r-n)+1-m}{2m^2}\psi\bigg(\frac{s+r}{m}\bigg)
\\ \hphantom{\sum_{\ell=1}^{m-1}I_{E}\bigg(\frac{\pi\ell}{m}\bigg) =}
{}+\frac{2\alpha_m(r+n)+1-m}{2m^2}\psi\bigg(\frac{s+2n+r}{m}\bigg)\bigg].
\end{gather*}
Finally, we find that the elliptic contribution is given by $\Xi_E=\mathscr{E}_E(s)-\mathscr{E}_E(a)$, where
\begin{align*}
\mathscr{E}_E(s)=
{}&\frac{1}{2s+2n-1}\sum_{j=1}^v\sum_{r=0}^{m_j-1}\bigg[\frac{2\alpha_{m_j}(r-n)+1-m_j}{2m_j^2}\psi \bigg(\frac{s+r}{m_j}\bigg)
\\
& +\frac{2\alpha_{m_j}(r+n)+1-m_j}{2m_j^2}\psi\bigg(\frac{s+2n+r}{m_j}\bigg)\bigg].
\end{align*}

\section[Contribution from parabolic elements and absolutely continuous spectrum]{Contribution from parabolic elements\\ and absolutely continuous spectrum}\label{a3}

 In this appendix, we compute the parabolic contribution $\Xi_P$ to the trace equation~\eqref{eq0903_6}.

\subsection{Contribution from parabolic elements}\label{par1}
In this section, we compute the term\vspace{-1ex}
\begin{gather*}
\Xi_{P, 1}= \iint_{F^Y} \sum_{\substack{\gamma\in\Gamma\\\gamma \;\text{is parabolic}}}k(\gamma z, z)\gamma'(z)^ny^{2n}{\rm d}\mu(z).
\end{gather*}
Again we first consider the case where $\Psi(u)=\Psi_{n,s}(u)$.
Recall that the Riemann surface $X$ has~$q$ cusps corresponding to the $q$ parabolic elements $\kappa_1, \dots, \kappa_q$. For each $1\leq i\leq q$, there is a~$\sigma_i\in{\rm PSL}(2,\mathbb{R})$ such that
\begin{gather*}
\sigma_i^{-1}\kappa_i\sigma_i=\begin{pmatrix} 1 & \pm1\\ 0 & 1\end{pmatrix}\!.
\end{gather*}
If $\gamma$ is a parabolic element of $\Gamma$, then there exists $\alpha\in \Gamma$, $1\leq j\leq q$ and a nonzero integer $\ell$ so that\vspace{-1ex}
\begin{gather*}
\alpha^{-1}\gamma\alpha=\kappa_j^{\ell}.
\end{gather*}
For an integer $\ell$, define\vspace{-1ex}
\begin{gather*}
T_{\ell}=\begin{pmatrix} 1 & \ell\\0 &1\end{pmatrix}\!.
\end{gather*}
Then\vspace{-1ex}
\begin{align*}
\Xi_{P, 1}&= \iint_{F^Y} \sum_{j=1}^q\sum_{\ell\neq 0}\sum_{\alpha\in \Gamma_{\kappa_j}\backslash\Gamma} k\big(\alpha\kappa_j^{\ell}\alpha^{-1} z,z\big)\big[\big(\alpha\kappa_j^{\ell}\alpha^{-1}\big)'(z)\big]^ny^{2n}{\rm d}\mu(z)
\\
&= \sum_{j=1}^q\sum_{\ell\neq 0} \sum_{\alpha\in \sigma_j^{-1}\Gamma_{\kappa_i}\sigma_j\backslash \sigma_j^{-1}\Gamma\sigma_j}\iint_{\sigma_j^{-1}(F^Y)} k\big(\alpha T_{\ell}\alpha^{-1} z,z\big)
\big[\big(\alpha T_{\ell}\alpha^{-1}\big)'(z)\big]^ny^{2n}{\rm d}\mu(z)
\\
&= \sum_{j=1}^q\sum_{\ell\neq 0} \iint_{\mathbb{H}_j^Y} k( T_{\ell} z,z)T_{\ell}'(z)^ny^{2n}{\rm d}\mu(z).
\end{align*}
 Here\vspace{-1ex}
 \begin{gather*}
 \mathbb{H}_j^Y=\mathbb{H}^Y\setminus \bigcup_{k\neq j} \bigcup_{\alpha \in \Gamma_{\kappa_j} \backslash \Gamma }\sigma_j^{-1}\big(\alpha\big(F_k^Y\big)\big),
 \end{gather*}
 where
 \begin{gather*}
 \mathbb{H}^Y=\big\{x+{\rm i}y\mid 0\leq x\leq 1, 0<y\leq Y\big\}.
 \end{gather*}
 Since the hyperbolic area of $\mathbb{H}\setminus \mathbb{H}^Y$ is $O\big(Y^{-1}\big)$, we find that
 \begin{gather*}
 \Xi_{P,1}=qI_P+O\big(Y^{-1}\big),
 \end{gather*}
 where
\begin{gather*}
I_P=\sum_{\ell\neq 0} \iint_{ \mathbb{H}^Y}k(T_{\ell} z,z)T_{\ell}'(z)^ny^{2n}{\rm d}\mu(z).
\end{gather*}
Now
\begin{gather*}
u(T_{\ell}z,z) = \frac{\ell^2}{4y^2},
\\
(-4)^n\frac{T_{\ell}'(z)^n}{(T_{\ell}z-\bar{z})^{2n}}= \frac{(-4)^n}{(\ell+2{\rm i} y)^{2n}}.
\end{gather*}
Hence, we have
\begin{align*}
I_P&=(-1)^n\sum_{\ell=1}^{\infty} \int_0^{Y} \Psi \bigg( \frac{\ell^2}{4y^2}\bigg)\bigg[\frac{(2y)^{2n}}{(l+2{\rm i} y)^{2n}}+\frac{(2y)^{2n}}{(l-2{\rm i}y)^{2n}}\bigg] \frac{{\rm d}y}{y^2}
\\
&=(-1)^n \sum_{\ell=1}^{\infty}\frac{2}{\ell}\int_0^{2Y/\ell}\Psi \bigg(\frac{1}{y^2}\bigg)\bigg[\frac{y^{2n}}{(1+{\rm i}y)^{2n}}
+\frac{y^{2n}}{(1-{\rm i}y)^{2n}}\bigg]\frac{{\rm d}y}{y^2}.
\end{align*}
As in \cite{Iwaniec2002}, if $F(y)$ is a well-behaved function, standard techniques in analytic number theory give
\begin{align*}
\sum_{\ell=1}^{\infty}\frac{2}{\ell} \int_0^{\frac{2Y}{\ell}}F(y)\,{\rm d}y
&=\int_{1^-}^{\infty} \int_0^{\frac{2Y}{v}}F(y)\,{\rm d}y \,{\rm d}
\bigg(\sum_{\ell\leq v}\frac{2}{\ell}\bigg)
=\int_0^{2Y}F(y)\Bigg[\int_{1^-}^{2Y/y}d\bigg(\sum_{\ell\leq v}\frac{2}{\ell}\bigg)\Bigg]{\rm d}y
\\
&=\int_0^{2Y}F(y)\bigg[2\log\frac{2Y}{y}+2\gamma+O\big(yY^{-1}\big)\bigg]{\rm d}y.
\end{align*}
Hence,
\begin{align*}
I_P&=(-1)^n \int_0^{2Y}\Psi \bigg(\frac{1}{y^2}\bigg)
\bigg[\frac{y^{2n}}{(1+{\rm i}y)^{2n}}+\frac{y^{2n}}{(1-{\rm i}y)^{2n}}\bigg] \bigg[2\log\frac{2Y}{y}+2\gamma+O\big(yY^{-1}\big)\bigg]\frac{{\rm d}y}{y^2}
\\
&= (\log(2Y)+\gamma)I_{P,0}+I_{P,1}+O\big(Y^{-1}\log Y\big),
\end{align*}
where
\begin{gather*}
I_{P,0}=2(-1)^n \int_0^{\infty}\Psi \bigg(\frac{1}{y^2}\bigg) \bigg[\frac{y^{2n}}{(1+{\rm i}y)^{2n}}
+\frac{y^{2n}}{(1-{\rm i}y)^{2n}}\bigg]\frac{{\rm d}y}{y^2},
\\
I_{P,1}=2(-1)^n \int_0^{\infty}\Psi \bigg(\frac{1}{y^2}\bigg) \bigg[\frac{y^{2n}}{(1+{\rm i}y)^{2n}}
+\frac{y^{2n}}{(1-{\rm i}y)^{2n}}\bigg] \log\frac{1}{y }\frac{{\rm d}y}{y^2}.
\end{gather*}
A straightforward computation gives
\begin{gather*}
I_{P,0}
=2(-1)^n\int_0^{\infty}\Psi (x^2)\bigg[\frac{1}{(x+{\rm i})^{2n}}+\frac{1}{(x-{\rm i})^{2n}}\bigg]{\rm d}x
=Q(0)
=\frac{1}{2s+2n-1}.
\end{gather*}
The computation of $I_{P,1}$ is more complicated:
 \begin{gather*}
I_{P,1}
=2(-1)^n\int_0^{\infty}\Psi (u^2)\bigg[\frac{1}{(u+{\rm i})^{2n}}
+\frac{1}{(u-{\rm i})^{2n}}\bigg]\log u\,{\rm d}u
\\ \hphantom{I_{P,1}}
{}=\frac{(-1)^n}{\pi}\int_0^{\infty} \sum_{k=0}^{\infty}\frac{\Gamma(s+k)\Gamma(s+2n+k)}{\Gamma(2s+2n+k)}\frac{1}{k!} \frac{1}{(u^2+1)^{k+s+2n}}
\\ \hphantom{I_{P,1}=\frac{(-1)^n}{\pi}\ \,}
{}\times\sum_{m=0}^n
\begin{pmatrix} 2n\\2m\end{pmatrix} u^{2m}(-1)^{n-m}\log u\,{\rm d}u
\\ \hphantom{I_{P,1}}
{}=\frac{1}{4\pi}\sum_{m=0}^n \begin{pmatrix} 2n\\2m\end{pmatrix} (-1)^m\sum_{k=0}^{\infty}\frac{\Gamma(s+k)\Gamma(s+2n+k)}{\Gamma(2s+2n+k)}\frac{1}{k!}
 \int_0^{\infty}\!\!\! \frac{u^m}{(u+1)^{k+s+2n}}\log u \frac{1}{\sqrt{u}}{\rm d}u.
\end{gather*}
This implies that
\begin{gather*}
I_{P,1}=G'(0),
\end{gather*}
where
\begin{gather*}
G(\alpha)=\frac{1}{4\pi}\sum_{m=0}^n \begin{pmatrix} 2n\\2m\end{pmatrix} (-1)^m\sum_{k=0}^{\infty}\frac{\Gamma(s+k)\Gamma(s+2n+k)}{\Gamma(2s+2n+k)}\frac{1}{k!}
\int_0^{\infty} \frac{u^{m+\alpha-\frac{1}{2}}}{(u+1)^{k+s+2n}}\,{\rm d}u.
\end{gather*}
From \cite{GradshteynRyzhik}, we find that
\begin{gather*}
 \int_0^{\infty} \frac{u^{m+\alpha-\frac{1}{2}}}{(u+1)^{k+s+2n}} {\rm d}u=\frac{\Gamma\left(m+\alpha+\frac{1}{2}\right) \Gamma\left(s+2n+k-m-\alpha-\frac{1}{2}\right)}{\Gamma(s+2n+k)}.
\end{gather*}
Hence,
\begin{gather*}
G(\alpha)=\frac{1}{4\pi}\sum_{m=0}^n \begin{pmatrix} 2n\\2m\end{pmatrix} (-1)^m\Gamma\left(m+\alpha+\frac{1}{2}\right)
\sum_{k=0}^{\infty}\frac{\Gamma(s+k)\Gamma\left(s+2n+k-m-\alpha-\frac{1}{2}\right)}{\Gamma(2s+2n+k)}\frac{1}{k!}.
\end{gather*}
Using the definition and identity of hypergeometric functions (see~\cite{Andrews}), we have
\begin{gather*}
\sum_{k=0}^{\infty}\frac{\Gamma(s+k)\Gamma\left(s+2n+k-m-\alpha-\frac{1}{2}\right)}{\Gamma(2s+2n+k )}\frac{1}{k!}
\\ \qquad
{}=\frac{\Gamma(s)\Gamma\left(s+2n-m-\alpha-\frac{1}{2}\right)}{\Gamma\left(2s+2n \right)}\,_2F_1\left(\begin{matrix} s, \;s+2n-m-\alpha-\frac{1}{2}\\ 2s+2n \end{matrix}\,; \,1\right)
\\ \qquad
{}=\frac{\Gamma(s)\Gamma\left(s+2n-m-\alpha-\frac{1}{2}\right)\Gamma\left(m+\alpha+\frac{1}{2}\right)}{\Gamma\left(s+2n \right)\Gamma\left(s+m+\alpha+\frac{1}{2}\right)}.
\end{gather*}
Now, we notice that
\begin{gather*}
\begin{pmatrix}
2n\\2m\end{pmatrix}
=\frac{(2n)(2n-1)\cdots (2n-2m+1)}{1\cdot 2 \cdots (2m-1)(2m)}
=\frac{(-n)_m\left(-n+\frac{1}{2}\right)_m}{\left(\frac{1}{2}\right)_m m!},
\\
(-1)^m \Gamma\bigg(s+2n-m-\alpha-\frac{1}{2}\bigg)
=(-1)^m\frac{\Gamma\left(s+2n-\alpha-\frac{1}{2}\right)}{\left(s+2n-\alpha-\frac{1}{2}-1\right)\cdots \left(s+2n-\alpha-\frac{1}{2}-m\right)}
\\ \hphantom{(-1)^m \Gamma\bigg(s+2n-m-\alpha-\frac{1}{2}\bigg)}
{}=\frac{\Gamma\left(s+2n-\alpha-\frac{1}{2}\right)}{\left(-s-2n+\alpha+\frac{1}{2}+1\right)\cdots
\left(-s-2n+\alpha+\frac{1}{2}+m\right)}
\\ \hphantom{(-1)^m \Gamma\bigg(s+2n-m-\alpha-\frac{1}{2}\bigg)}
{}=\frac{\Gamma\left(s+2n-\alpha-\frac{1}{2}\right)}{\left(-s-2n+ \alpha+\frac{3}{2}\right)_m}.
\end{gather*}
This gives
\begin{gather*}
G(\alpha)=\frac{1}{4\pi}\frac{\Gamma\left(\alpha+\frac{1}{2}\right)^2\Gamma(s) \Gamma\left(s+2n-\alpha-\frac{1}{2}\right)}{\Gamma(s+2n ) \Gamma\left(s+\alpha+\frac{1}{2}\right)}H(\alpha),
\end{gather*}
where
\begin{align*}
H(\alpha)&=\sum_{m=0}^n \frac{(-n)_m\left(-n+\frac{1}{2}\right)_m}{\left(\frac{1}{2}\right)_m m!}\frac{\left(\alpha+\frac{1}{2}\right)_m \left(\alpha+\frac{1}{2}\right)_m }{\left(-s-2n+ \alpha+\frac{3}{2}\right)_m\left(s+\alpha+\frac{1}{2}\right)_m}
\\
&=\,_4F_3\left(\begin{matrix} -n, \,-n+\frac{1}{2},\, \alpha+\frac{1}{2}, \,\alpha+\frac{1}{2}\\ \frac{1}{2}, \,-s-2n+\alpha+ \frac{3}{2},\,s+\alpha+\frac{1}{2}\end{matrix};\;1\right)\!.
\end{align*}
This is a balanced hypergeometric series. According to \cite[Theorem~3.3.3]{Andrews}, we find that
\begin{align*}
H(\alpha)={}&\frac{\Gamma(\alpha+1-s)\Gamma(\alpha+s+2n)
\Gamma\left(\alpha-s-2n+\frac{3}{2}\right)\Gamma\left(\alpha+s+\frac{1}{2}\right)} {\Gamma(\alpha+1-s-n)\Gamma(\alpha+s+n)\Gamma\left(\alpha-s-n + \frac{3}{2}\right)\Gamma\left(\alpha+s+n+\frac{1}{2}\right)}
\\
&\times \,_4F_3\left(\begin{matrix} -n, \,-n+\frac{1}{2},\, -\alpha, \,-\alpha\\ \frac{1}{2}, \,s-\alpha,\,1-s-2n-\alpha \end{matrix}\,;\;1\right)\!.
\end{align*}
As a function of $\alpha$,
\begin{gather*}
{}_4F_3\left(\begin{matrix} -n, \,-n+\frac{1}{2},\, -\alpha, \,-\alpha\\ \frac{1}{2}, \,s-\alpha,\,1-s-2n-\alpha \end{matrix}\,;\;1\right)=1+O\big(\alpha^2\big) \qquad
\text{as}\quad \alpha\rightarrow 0.
\end{gather*}
Hence, $I_{P,1}=G'(0)=L'(0)$, where $L(\alpha)$ is given by
\begin{align*}
 L'(\alpha)=\frac{1}{4\pi}\frac{\Gamma\big(\alpha\!+\frac{1}{2}\big)^2\Gamma(s) \Gamma\big(s\!+2n\!-\alpha-\frac{1}{2}\big)\Gamma(\alpha+1-s)\Gamma(\alpha+s+2n) \Gamma\big(\alpha-s-2n+ \frac{3}{2}\big)}{\Gamma(s+2n ) \Gamma(\alpha+1-s-n)\Gamma(\alpha+s+n)\Gamma\big(\alpha-s-n + \frac{3}{2}\big)\Gamma\big(\alpha+s+n+\frac{1}{2}\big)}.
\end{align*}
It is then straightforward to obtain
\begin{align*}
 I_{P,1}={}&\frac{1}{4\pi}\frac{\Gamma\left(\frac{1}{2}\right)^2\Gamma(s) \Gamma\left(s+2n -\frac{1}{2}\right)\Gamma(1-s)\Gamma(s+2n)\Gamma\left(-s-2n+ \frac{3}{2}\right)}{\Gamma\left(s+2n \right) \Gamma(1-s-n)\Gamma(s+n)\Gamma\left(-s-n + \frac{3}{2}\right)\Gamma\left( s+n+\frac{1}{2}\right)}
 \\
 &\times \bigg\{2\psi\bigg(\frac{1}{2}\bigg) -\psi\bigg(s+2n -\frac{1}{2}\bigg)+\psi(1-s) +\psi(s+2n)+\psi\bigg({-}s-2n+ \frac{3}{2}\bigg)
 \\
 &-\psi(1-s-n)-\psi(s+n)-\psi\bigg({-}s-n + \frac{3}{2}\bigg)-\psi\bigg(s+n+\frac{1}{2}\bigg)\bigg\}.
\end{align*}
Using the fact that $\Gamma(z+1)=z\Gamma(z)$, we find that
\begin{gather*}
\Gamma\bigg(s+n+\frac{1}{2}\bigg)= \bigg(s+n-\frac{1}{2}\bigg)\Gamma\bigg(s+n-\frac{1}{2}\bigg),
\end{gather*}
and
\begin{gather*}
\psi\bigg(s+n+\frac{1}{2}\bigg)= \frac{1}{s+n-\frac{1}{2}}+\psi\bigg(s+n-\frac{1}{2}\bigg).
\end{gather*}
On the other hand, using
\begin{gather*}
\Gamma(z)\Gamma(1-z)=\frac{\pi}{\sin \pi z},
\end{gather*}
we find that
\begin{gather*}
\frac{\Gamma(s)\Gamma(1-s)}{\Gamma(1-s-n)\Gamma(s+n)}=\frac{\sin\pi(s+n)}{\sin \pi s}=(-1)^n,
\\[.5ex]
\frac{\Gamma(s+2n -\frac{1}{2}) \Gamma\left(-s-2n+ \frac{3}{2}\right)}{ \Gamma\left( -s-n + \frac{3}{2}\right)\Gamma\left( s+n-\frac{1}{2}\right)}
=\frac{\sin\pi\left(s+n-\frac{1}{2}\right)}{\sin \pi(s+2n -\frac{1}{2})}=(-1)^n,
\\[.5ex]
\psi(s)-\psi(1-s)+\psi(1-s-n)-\psi(s+n)=0,
\\
\psi\bigg(s+n-\frac{1}{2}\bigg)-\psi\bigg({-}s-2n+\frac{3}{2}\bigg)+\psi\bigg({-}s-n+\frac{3}{2}\bigg)
-\psi\bigg(s+n-\frac{1}{2}\bigg)=0.
\end{gather*}
Together with
\begin{gather*}
\Gamma\bigg(\frac{1}{2}\bigg)=\sqrt{\pi},\qquad \psi\le\bigg(\frac{1}{2}\bigg)=-\gamma-2\log 2,
\end{gather*}
we have
\begin{gather*}
\frac{1}{4\pi}\frac{\Gamma\left(\frac{1}{2}\right)^2\Gamma(s) \Gamma\left(s+2n -\frac{1}{2}\right) \Gamma( 1-s) \Gamma\left(-s-2n+ \frac{3}{2}\right)}{\Gamma(1-s-n)\Gamma(s+n) \Gamma\left(-s-n+\frac{3}{2}\right)\Gamma\left(s+n+\frac{1}{2}\right)}=\frac{1}{2(2s+2n-1)},
\end{gather*}
and
\begin{align*}
I_{P,1}={}&\frac{1}{2(2s+2n-1)}\bigg\{\psi(s)+\psi(s+2n)-4\log 2-2\psi\bigg(s+n+\frac{1}{2}\bigg)-2\psi(s+n)\bigg\}
\\
&-\frac{\gamma}{2s+2n-1}+\frac{1}{(2s+2n-1)^2}.
\end{align*}
Collecting the various contributions, we find that
\begin{align*}
I_P=&\frac{1}{2s+2n-1}\log Y+\frac{1}{(2s+2n-1)^2}
\\
&+\frac{1}{2(2s+2n-1)}\left\{\psi(s)+\psi(s+2n)-2\log 2-2\psi\left(s+n+\frac{1}{2}\right)-2\psi(s+n)\right\}.
\end{align*}
Therefore, $\Xi_{P,1}=\mathscr{E}_{P,1}(s)-\mathscr{E}_{P,1}(a)$, where
\begin{align*}
\mathscr{E}_{P,1}(s)={}&\frac{q}{2s+2n-1}\log Y+\frac{q}{(2s+2n-1)^2}
+\frac{q}{2(2s+2n-1)}
\\
&\times \bigg\{\psi(s)+\psi(s+2n)-2\log 2-2\psi\bigg(s+n+\frac{1}{2}\bigg)-2\psi(s+n)\bigg\}+O\big(Y^{-1}\big).
\end{align*}

\subsection{Contribution of the absolutely continuous spectrum}\label{par2}

In this section, we want to compute
\begin{gather*}
J=\frac{1}{4\pi}\sum_{j=1}^q\int_{-\infty}^{\infty}h(r)
\iint_{F^Y}\bigg|E_j\bigg(z,\frac{1}{2}+{\rm i}r;n\bigg)\bigg|^2y^{2n}{\rm d}\mu(z)\,{\rm d}r,
\end{gather*}
where
\begin{gather*}
h(r)=\widetilde{\Lambda}\left(\frac{1}{4}+r^2\right) =\frac{1}{ r^2+\left(s+n-\frac{1}{2}\right)^2}-\frac{1}{ r^2+\left(a+n-\frac{1}{2}\right)^2}.
\end{gather*}
From Appendix \ref{maass}, we find that
\begin{gather*}
\sum_{j=1}^q
\iint_{F^Y}\bigg|E_j\bigg(z,\frac{1}{2}+{\rm i}r;n\bigg)\bigg|^2y^{2n}{\rm d}\mu(z)
\\ \qquad
{}=\frac{1}{2{\rm i}r}\operatorname{Tr}\bigg[\Phi\bigg(\frac{1}{2}-{\rm i}r\bigg)Y^{2{\rm i}r}-\Phi\bigg(\frac{1}{2}+{\rm i}r\bigg)Y^{-2{\rm i}r}\bigg]+2q\log Y-\frac{\varphi'}{\varphi}\bigg(\frac{1}{2}+{\rm i}r\bigg).
\end{gather*}
Therefore,
\begin{gather*}
J=J_1+J_2+J_3,
\end{gather*}
where
\begin{gather*}
J_1= \frac{q\log Y}{2\pi} \int_{-\infty}^{\infty} h(r)\,{\rm d}r,
\\
J_2= \frac{1}{8\pi{\rm i} }\int_{-\infty}^{\infty} \frac{h(r)}{r}\operatorname{Tr}\bigg[\Phi\bigg(\frac{1}{2}-{\rm i}r\bigg)Y^{2{\rm i}r}-\Phi\bigg(\frac{1}{2}+{\rm i}r\bigg)Y^{-2{\rm i}r}\bigg]{\rm d}r,
\\
J_3= -\frac{1}{4\pi }\int_{-\infty}^{\infty} h(r) \frac{\varphi'}{\varphi}\bigg(\frac{1}{2}+{\rm i}r\bigg){\rm d}r.
\end{gather*}
Since
\begin{gather*}
\frac{1}{2\pi}\int_{-\infty}^{\infty} h(r)\,{\rm d}r=g(0),
\end{gather*}
we find that
\begin{gather*}
J_1=q\log Y\bigg(\frac{1}{2s+2n-1}-\frac{1}{2a+2n-1}\bigg).
\end{gather*}
As in \cite{Iwaniec2002},
\begin{align*}
J_2&=\frac{1}{8\pi{\rm i}} \int_{-\infty}^{\infty}\frac{h(r)}{r}\operatorname{Tr}\bigg[\Phi\bigg(\frac{1}{2}-{\rm i}r\bigg)Y^{2{\rm i}r}-\Phi\bigg(\frac{1}{2}+{\rm i}r\bigg)Y^{-2{\rm i}r}\bigg]{\rm d}r
\\
&=\frac{1}{4}\operatorname{Tr}\Phi\bigg(\frac{1}{2}\bigg)h(0)+O\big(Y^{-2\varepsilon}\big)
\\
&=\bigg[\frac{1}{(2s+2n-1)^2}-\frac{1}{(2a+2n-1)^2}\bigg]
\operatorname{Tr}\Phi\bigg(\frac{1}{2}\bigg)+O\big(Y^{-2\varepsilon}\big),
\end{align*}
where $\varepsilon$ is an arbitrary positive number.
Collecting everything together, we find that
\begin{align*}
J={}&q\log Y\bigg(\frac{1}{2s+2n-1}-\frac{1}{2a+2n-1}\bigg)
+\bigg[\frac{1}{(2s+2n-1)^2}-\frac{1}{(2a+2n-1)^2}\bigg]
\operatorname{Tr}\Phi\bigg(\frac{1}{2}\bigg)
\\
&-\frac{1}{4\pi }\int_{-\infty}^{\infty} h(r) \frac{\varphi'}{\varphi}
\bigg(\frac{1}{2}+{\rm i}r\bigg){\rm d}r+O\big(Y^{-2\varepsilon}\big).
\end{align*}
When $h(r)=\frac{1}{r^2+\alpha^2}$ for some $\alpha$ with $\operatorname{Re}\alpha>0$, we can follow the method of \cite{Venkov1982} to compute the integral
\begin{gather*}
J_3= -\frac{1}{4\pi }\int_{-\infty}^{\infty} h(r) \frac{\varphi'}{\varphi}\left(\frac{1}{2}+{\rm i}r\right){\rm d}r.
\end{gather*}
From \eqref{eqphi} and the results for the case $n=0$, we can deduce that $\varphi(s)$ is holomorphic on the half plane $\operatorname{Re}s\geq\frac{1}{2}$ except for
for a finite number of poles on $\left(\frac{1}{2}, 1\right]$, with order not larger than~$q$. From the relation $\varphi(s)\varphi(1-s)=1$, we then deduce that on the half plane $\operatorname{Re}s<\frac{1}{2}$, $\varphi(s)$ can only have zeros on $\left[0, \frac{1}{2}\right)$. Using these and the residue theorem, we find that
\begin{gather}
\frac{1}{2\pi}\int_{-\infty}^{\infty}\frac{1}{r^2+\alpha^2} \frac{\varphi'}{\varphi}\bigg(\frac{1}{2}+{\rm i}r\bigg){\rm d}r
=\frac{1}{2s+2n-1}\frac{\varphi'(s+n)}{\varphi(s+n)}
-\sum_{\substack{\text{$\rho$ is a pole of $\varphi(s)$}\\ \operatorname{Re}\rho<\frac{1}{2}}}\frac{\text{order}\,(\rho)}{\alpha^2-\left(\rho-\frac{1}{2}\right)^2}
\nonumber
\\ \hphantom{\frac{1}{2\pi}\int_{-\infty}^{\infty}\frac{1}{r^2+\alpha^2} \frac{\varphi'}{\varphi}\bigg(\frac{1}{2}+{\rm i}r\bigg){\rm d}r=}
{}+ \sum_{\substack{\text{$\sigma$ is a pole of $\varphi(s)$}\\\sigma \in\left(\frac{1}{2},1\right]}}\frac{\text{order}\,(\sigma)}{\alpha^2-\left(\sigma-\frac{1}{2}\right)^2}+c.
\label{scattering}
\end{gather}
Here $\text{order}(z)$ is the order of pole of $\varphi(s)$ at $s=z$, $c$ is a constant independent of $\alpha$.

We are interested in the particular case where $\alpha=s+n-\frac{1}{2}$. The following result is needed when we want to discuss the dimension of the space of holomorphic $n$-differentials.
\begin{Proposition}\label{polescatter}
Given a positive integer $n$, let $\Sigma(s)$ be the function
\begin{gather*}
\Sigma(s)=\frac{1}{2\pi }\int_{-\infty}^{\infty}\frac{1}{r^2+\left(s+n-\frac{1}{2}\right)^2} \frac{\varphi'}{\varphi}\bigg(\frac{1}{2}+{\rm i}r\bigg){\rm d}r.
\end{gather*}
Then the residue of $\Sigma(s)$ at $s=0$ is 0.
\end{Proposition}
\begin{proof}
According to \eqref{scattering}, $\Sigma(s)$ is a sum of three terms. We denote these three terms by $\Sigma_1(s)$, $\Sigma_2(s)$ and $\Sigma_3(s)$ respectively.

We need to discuss the case where $n=1$ and $n\geq 2$ separately.

When $n=1$, the residue of $\Sigma_1(s)$ at $s=0$ is equal to $-n_0$, where $n_0$ is the order of pole of~$\varphi(s)$ at $s=1$, which is equal to the order of zero of $\varphi(s)$ at $s=0$. Since $\varphi(s)$ does not have pole at $s=0$, the residue of $\Sigma_2(s)$ at $s=0$ is 0. Since $\varphi(s)$ has a pole of order $n_0$ at $s=1$, the residue of $\Sigma_3(s)$ at $s=0$ is $n_0$. This shows that the residue of $\Sigma(s)$ at $s=0$ is 0.

When $n\geq 2$, the residue of $\Sigma_1(s)$ is equal to $\frac{1}{2n-1}$ times the order of zero of $\varphi(s)$ at $s=n$. Since $\varphi(s)$ does not have pole when $\operatorname{Re}s>1$, the residue of $\Sigma_3(s)$ at $s=0$ is $0$. For $\Sigma_2(s)$, the order of zero of $\varphi(s)$ at $s=n$ is equal to the order of pole of $\varphi(s)$ at $1-n$. Therefore, the residue of $\Sigma_2(s)$ at $s=0$ is the negative of the residue of $\Sigma_1(s)$ at $s=0$. This proves the assertion that the residues of $\Sigma(s)$ at $s=0$ is 0.
\end{proof}

\subsection[The term Xi P]{The term $\boldsymbol{\Xi_P}$}
Collecting the results from Appendices~\ref{par1} and~\ref{par2}, we find that the parabolic contribution
\begin{gather*}
\Xi_P=\lim_{Y\rightarrow\infty}\Bigg\{\iint_{F^Y} \sum_{\substack{\gamma\in\Gamma\\\gamma \text{is parabolic}}}k(\gamma z,z)\gamma'(z)^ny^{2n}{\rm d}\mu(z)
\\ \hphantom{\Xi_P=\lim_{Y\rightarrow\infty}\Bigg\{}
{}-\frac{1}{4\pi}\sum_{j=1}^q\int_{-\infty}^{\infty}\widetilde{\Lambda}(r)
\iint_{F^Y}\bigg|E_j\bigg(z,\frac{1}{2}+{\rm i}r;n\bigg)\bigg|^2y^{2n}{\rm d}\mu(z)\,{\rm d}r\Bigg\}
\end{gather*}
is given by $\Xi_P=\mathscr{E}_P(s)-\mathscr{E}_P(a)$, where
\begin{align*}
\mathscr{E}_P(s)={}&\frac{1}{(2s+2n-1)^2} \bigg[q-\operatorname{Tr} \Phi\bigg(\frac{1}{2}\bigg) \bigg]+\frac{1}{4\pi}\int_{-\infty}^{\infty}
\frac{1}{r^2+\left(s+n-\frac{1}{2}\right)^2}\frac{\varphi'}{\varphi}
\bigg(\frac{1}{2}+{\rm i}r\bigg){\rm d}r
\\
&+\frac{q}{2(2s+2n-1)}\bigg\{\psi(s)+\psi(s+2n)-2\log 2-2\psi\bigg(s+n+\frac{1}{2}\bigg)-2\psi(s+n)\bigg\}.
\end{align*}

\section{Maass--Selberg relation}\label{maass}
The Maass--Selberg relation is very important in the computation of the contribution from the absolutely continuous spectrum. The results in this part follows almost the same as the case $n=0$ as presented in \cite{Iwaniec2002}.

Given a function $f$ in $H_n^2(\Gamma)$ such that $\Delta_n f=\lambda f$, it has the following Fourier expansion around the cusp associated to $\kappa_j$.
\begin{gather*}f(\sigma_j z)\sigma_j'(z)^n=\sum_{k=-\infty}^{\infty} f_{k}^{(j)}(y){\rm e}^{2\pi{\rm i} kx}.\end{gather*}
We define
\begin{gather*}
f^{Y}(z)=\begin{cases} f(z), & z\in F^Y,
\\
f(z)-f_0^{(j)}\big(\sigma_j^{-1}z\big), & z\in F_j^Y.\end{cases}
\end{gather*}
\begin{Theorem}\label{theorem_A1}
If $f$ and $g$ are functions in $H_n^2(\Gamma)$ such that
\begin{gather*}
\Delta_n f= \lambda_1 f,\qquad\Delta_n g=\lambda_2 g.
\end{gather*}
Then
\begin{gather*}
(\lambda_1-\lambda_2)\iint_{F}f^Y\overline{g^Y}y^{2n-2}{\rm d}x\,{\rm d}y=-Y^{2n}\sum_{j=1}^q\Big(f_0^{(j)}(Y)\overline{g_0^{(j)\prime}(Y)}-f_0^{(j)\prime}(Y)
\overline{g_0^{(j)}(Y)}\Big).
\end{gather*}
\end{Theorem}

\begin{proof}
\begin{align*}
(\lambda_1-\lambda_2)\iint_{F^Y}f\bar{g}y^{2n-2}{\rm d}x\,{\rm d}y
={}&\iint_{F^Y}\Delta_nf\bar{g}y^{2n-2}{\rm d}x\,{\rm d}y-\iint_{F^Y}f\overline{\Delta_n g}y^{2n-2}{\rm d}x\,{\rm d}y
 \\
={}&-2{\rm i}\iint_{F^Y}\bigg\{
\bigg(\frac{\pa}{\pa z}y^{2n}\frac{\pa}{\pa\bar{z}}f\bigg)\bar{g}-f \bigg(\frac{\pa}{\pa \bar{z}}y^{2n}\frac{\pa}{\pa z}\bar{g}\bigg)\bigg\}{\rm d}z\wedge {\rm d}\bar{z}
\\
={}&-2{\rm i}\int_{\pa F^Y}\bigg(
 y^{2n}\frac{\pa f}{\pa\bar{z}} \bar{g}{\rm d}\bar{z}+f y^{2n}\frac{\pa \bar{g}}{\pa z} {\rm d}z\bigg)
 \\
={}&\sum_{j=1}^q -2{\rm i}\!\int_{\sigma_j^{-1}\pa F_j^Y}\!\!\bigg(
 y^{2n}\frac{\pa f}{\pa\bar{z}}(\sigma_j z) \sigma_j'(z)^n\overline{\sigma_j'(z)}\;\overline{g(\sigma_j z)\sigma'(z)^n}\;{\rm d}\bar{z}
 \\
&+f(\sigma_j z)\sigma_j'(z)^n y^{2n}\overline{\frac{\pa g}{\pa \bar{z}}(\sigma_j z) \sigma_j'(z)^n\overline{\sigma_j'(z)}}\;{\rm d}z\bigg)
 \\
={}&Y^{2n}\sum_{j=1}^q 2i\int_0^1\bigg(
 \frac{\pa f}{\pa\bar{z}}(\sigma_j z) \sigma_j'(z)^n\overline{\sigma_j'(z)}\;\overline{g(\sigma_j z)\sigma'(z)^n}\;{\rm d}x
 \\
 & +f(\sigma_j z)\sigma_j'(z)^n \overline{\frac{\pa g}{\pa \bar{z}}(\sigma_j z) \sigma_j'(z)^n\overline{\sigma_j'(z)}}\;{\rm d}x\bigg).
\end{align*}
Notice that
\begin{gather*}
\frac{\pa f}{\pa\bar{z}}(\sigma_j z) \sigma_j'(z)^n\overline{\sigma_j'(z)}= \sum_{k=-\infty}^{\infty}\pi{\rm i} kf_{k}^{(j)}(y){\rm e}^{2\pi{\rm i} kx}
+\frac{\rm i} {2}\sum_{k=-\infty}^{\infty}\frac{\pa f_{k}^{(j)}}{\pa y}(y){\rm e}^{2\pi{\rm i} kx}.
\end{gather*}
Hence,
\begin{gather*}
(\lambda_1-\lambda_2)\iint_{F^Y}f\bar{g}y^{2n-2}{\rm d}x\,{\rm d}y= -Y^{2n}\sum_{j=1}^q\sum_{k=-\infty}^{\infty}\bigg(\frac{\pa f_{k}^{(j)}}{\pa y}(Y)\overline{g_{k}^{(j)}(Y)}-f_{k}^{(j)}(Y)\overline{\frac{\pa g_{k}^{(j)}}{\pa y}(Y)}\bigg).
\end{gather*}
Applying $\Delta_n f=\lambda f$, where $\lambda=-(s-n)(s+n-1)$, to the Fourier expansion, we find that
\begin{gather*}
y^2\sum_{k=-\infty}^{\infty} (2\pi k)^2f_k^{(j)}(y){\rm e}^{2\pi{\rm i} k x}-y^2\sum_{k=-\infty}^{\infty} \frac{\pa^2 f_k^{(j)}}{\pa y^2}(y){\rm e}^{2\pi{\rm i} k x}
-2ny\sum_{k=-\infty}^{\infty} 2\pi kf_k^{(j)}(y){\rm e}^{2\pi{\rm i} k x}
\\ \qquad
{}-2ny\sum_{k=-\infty}^{\infty} \frac{\pa f_k^{(j)}}{\pa y}(y){\rm e}^{2\pi{\rm i} k x}
=-(s-n)(s+n-1)\sum_{k=-\infty}^{\infty} f_k^{(j)}(y){\rm e}^{2\pi{\rm i} k x}.
\end{gather*}
Hence,
\begin{gather*}
y^2\frac{\pa^2 f_k^{(j)}}{\pa y^2}(y)+2ny\frac{\pa f_k^{(j)}}{\pa y}(y)- \left(4\pi k^2y^2-4\pi n k y+(s-n)(s+n-1)\right)f_k^{(j)}=0.
\end{gather*}
This implies that
\begin{gather*}
f_0^{(j)}(y)=\alpha_0y^{s-n}+\beta_0y^{1-s-n}.
\end{gather*}
On the other hand,
\begin{gather*}
\frac{\pa}{\pa y} y^{2n}\bigg(\frac{\pa f_{k}^{(j)}}{\pa y}\overline{g_{k}^{(j)}}-f_{k}^{(j)}\overline{\frac{\pa g_{k}^{(j)}}{\pa y}}\bigg) =y^{2n-2}[(s_1\!-n)(s_1+n\!-1)-(s_2\!-n)(s_2+n\!-1)]f_k^{(j)}\overline{g_k^{(j)}}.
\end{gather*}
Hence, if $k\neq 0$,
\begin{gather*}
-Y^{2n}\sum_{j=1}^q\sum_{k\neq 0}\bigg(\frac{\pa f_{k}^{(j)}}{\pa y}(Y)\overline{g_{k}^{(j)}(Y)}-f_{k}^{(j)}(Y)\overline{\frac{\pa g_{k}^{(j)}}{\pa y}(Y)}\bigg)
\\ \qquad
{}=-(\lambda_1-\lambda_2)\sum_{j=1}^q\sum_{k\neq 0}\int_Y^{\infty} y^{2n-2}f_k^{(j)}\overline{g_k^{(j)}}{\rm d}y
=-(\lambda_1-\lambda_2)\sum_{j=1}^q\iint_{F_j^Y}f^Y\bar{g}^Yy^{2n-2}{\rm d}x\,{\rm d}y.
\end{gather*}
It follows that
\begin{align*}
(\lambda_1-\lambda_2)\iint_{F^Y}f\bar{g}y^{2n-2}{\rm d}x\,{\rm d}y={}&-Y^{2n}\sum_{j=1}^q
\bigg(\frac{\pa f_{0}^{(j)}}{\pa y}(Y)\overline{g_{0}^{(j)}(Y)}-f_{0}^{(j)}(Y)\overline{\frac{\pa g_{0}^{(j)}}{\pa y}(Y)}\bigg)
\\
&- (\lambda_1-\lambda_2)\sum_{j=1}^q\iint_{F_j^Y}f^Y\bar{g}^Yy^{2n-2}{\rm d}x\,{\rm d}y.
\end{align*}
Hence,
\begin{gather*}
(\lambda_1-\lambda_2)\iint_{F}f^Y\overline{g^Y}y^{2n-2}{\rm d}x\,{\rm d}y=-Y^{2n}\sum_{j=1}^q\Big(f_0^{(j)}(Y)\overline{g_0^{(j)\prime}(Y)}-f_0^{(j)\prime}(Y)
\overline{g_0^{(j)}(Y)}\Big).
\tag*{\qed}
\end{gather*}
\renewcommand{\qed}{}
\end{proof}

\begin{Theorem}[Maass--Selberg relation] If $s_1$ and $s_2$ are regular points of the Eisenstein series $E_i(z,s_1;n)$ and $E_i(z,s_2;n)$, $s_1\neq \bar{s}_2$ and $s_1+\bar{s}_2\neq 1$, then
\begin{gather*}
\iint_F E_i^Y(z,s_1;n)\overline{E_j^Y(z,s_2;n)}y^{2n-2}{\rm d}x\,{\rm d}y
\\ \qquad
{}=\delta_{ij}\frac{1}{s_1+\overline{s_2}-1}Y^{s_1+\overline{s_2}-1}
+\frac{1}{\overline{s_2}-s_1}\varphi_{ij}(s_1)Y^{\overline{s_2}-s_1}
+\frac{1}{s_1-\overline{s_2}}\overline{\varphi_{ji}(s_2)}Y^{s_1-\overline{s_2}}
\\ \qquad \phantom{=}
-\frac{1}{s_1+\overline{s_2}-1}\sum_{k=1}^q\varphi_{ik}(s_1)
\overline{\varphi_{jk}(s_2)}Y^{1-s_1-\overline{s_2}}.
\end{gather*}
\end{Theorem}

\begin{proof}
By Theorem \ref{theorem_A1}, we have
\begin{gather*}
\iint_F E_i^Y(z,s_1;n)\overline{E_j^Y(z,s_2;n)}y^{2n-2}{\rm d}x\,{\rm d}y
=\frac{Y^{2n}}{-(s_1-n)(s_1+n-1)+(\overline{s_2}-n)(\overline{s_2}+n-1)}
\\ \qquad
{}\times\sum_{k=1}^q\Big([\delta_{ik}Y^{s_1-n}\!+\!\varphi_{ik}(s_1)Y^{1-s_1-n}]
\overline{[\delta_{jk}(s_2\!-n)Y^{s_2\!n\!-1}+(1\!-s_2\!-n) \varphi_{jk}(s_2)Y^{-s_2\!-n}]}
\\ \qquad\phantom{\times}
{}-[\delta_{ik}(s_1-n)Y^{s_1-n-1}+(1-s_1-n)\varphi_{ik}(s_1)Y^{-s_1-n}]
\overline{[\delta_{jk}Y^{s_2-n}+ \varphi_{jk}(s_2)Y^{1-s_2-n}]} \Big)
\\ \qquad
{}=\delta_{ij}\frac{1}{s_1+\overline{s_2}-1}Y^{s_1+\overline{s_2}-1}
+\frac{1}{\overline{s_2}-s_1}\varphi_{ij}(s_1)Y^{\overline{s_2}-s_1}
+\frac{1}{s_1-\overline{s_2}}\overline{\varphi_{ji}(s_2)}Y^{s_1-\overline{s_2}}
\\ \qquad\phantom{=}
{}-\frac{1}{s_1+\overline{s_2}-1}\sum_{k=1}^q\varphi_{ik}(s_1)
\overline{\varphi_{jk}(s_2)}Y^{1-s_1-\overline{s_2}}.
\tag*{\qed}
\end{gather*}
\renewcommand{\qed}{}
\end{proof}

\begin{Theorem}As $Y\rightarrow\infty$,
\begin{gather*}
\sum_{j=1}^q
\iint_{F^Y}\bigg|E_j\bigg(z,\frac{1}{2}+{\rm i}r;n\bigg) \bigg|^2y^{2n}{\rm d}\mu(z)
\\ \qquad
{}=\frac{1}{2{\rm i}r}\operatorname{Tr}\bigg[\Phi\bigg(\frac{1}{2}-{\rm i}r\bigg)Y^{2{\rm i}r}-\Phi\bigg(\frac{1}{2}+{\rm i}r\bigg)Y^{-2{\rm i}r}\bigg]+2q\log Y-\frac{\varphi'}{\varphi}\bigg(\frac{1}{2}+{\rm i}r\bigg).
\end{gather*}
\end{Theorem}

\begin{proof}
From the proof of the previous theorem, we find that
\begin{align*}
\iint_{F^Y}\bigg|E_j\bigg(z,\frac{1}{2}+{\rm i}r;n\bigg) \bigg|^2y^{2n}{\rm d}\mu(z)
={}&\iint_{F}\bigg|E_j^Y\bigg(z,\frac{1}{2}+{\rm i}r;n\bigg) \bigg|^2y^{2n}{\rm d}\mu(z)
\\
&+\text{exponentially decaying terms}.
\end{align*}
Setting $s_1=s_2=\sigma+{\rm i}r$ in the previous theorem, we find that
\begin{align*}
\iint_{F}\bigg|E_j^Y\bigg(z,\frac{1}{2}+{\rm i}r;n\bigg)\bigg|^2y^{2n}{\rm d}\mu(z)
={} &\frac{1}{2\sigma-1}Y^{2\sigma-1}-\frac{1}{2{\rm i}r}\varphi_{jj}(\sigma+{\rm i}r)Y^{-2{\rm i}r}
\\
&+\frac{1}{2{\rm i}r}\overline{\varphi_{jj}(\sigma+{\rm i}r)}Y^{2{\rm i}r}
\\
&-\frac{1}{2\sigma-1}\sum_{k=1}^q\varphi_{jk}(\sigma+{\rm i}r)\overline{\varphi_{jk}(\sigma+{\rm i}r)}Y^{1-2\sigma}.
\end{align*}
We want to take the limit when $\sigma\rightarrow 1/2$. We find that as $\sigma\rightarrow 1/2$,
\begin{gather*}
Y^{2\sigma-1}=1+(2\sigma-1)\log Y+ O\big((2\sigma-1)^2\big),
\\
\varphi_{jk}(\sigma+{\rm i}r)=\varphi_{jk}\bigg(\frac{1}{2}+{\rm i}r\bigg) +\bigg(\sigma-\frac{1}{2}\bigg)\varphi'_{jk}\bigg(\frac{1}{2}+{\rm i}r\bigg)+ O\big((2\sigma-1)^2\big).
\end{gather*}
Hence,
\begin{gather*}
\sum_{j=1}^q\iint_{F}\bigg|E_j^Y\bigg(z,\frac{1}{2}+{\rm i}r;n\bigg) \bigg|^2y^{2n}{\rm d}\mu(z)
\\ \qquad
=\frac{1}{2{\rm i}r}\operatorname{Tr}\bigg[\Phi\bigg(\frac{1}{2}-{\rm i}r\bigg)Y^{2{\rm i}r}-\Phi\bigg(\frac{1}{2}+{\rm i}r\bigg)Y^{-2{\rm i}r}\bigg]+\frac{q}{2\sigma-1}\bigg(1+(2\sigma-1)\log Y\bigg)
\\ \qquad\phantom{=}
{}-\frac{1}{2\sigma-1}\sum_{j=1}^q\sum_{k=1}^q\bigg[\varphi_{jk}\bigg(\frac{1}{2}+{\rm i}r\bigg)+
\bigg(\sigma-\frac{1}{2}\bigg)\varphi'_{jk}\bigg(\frac{1}{2}+{\rm i}r\bigg)\bigg]
\\ \qquad\phantom{=-}
{}\times
\bigg[\varphi_{kj}\bigg(\frac{1}{2}-{\rm i}r\bigg)+\bigg(\sigma-\frac{1}{2}\bigg)\varphi'_{kj}\bigg(\frac{1}{2}-{\rm i}r\bigg)\bigg]\bigg(1-(2\sigma-1)\log Y\bigg)+O(2\sigma-1).
\end{gather*}
Now the matrix $\Phi(s)=\big[\varphi_{ij}(s)\big]$ satisfies
\begin{gather*}
\Phi(s)\Phi(1-s)=I.
\end{gather*}
Hence, for each $1\leq j\leq q$,
\begin{gather*}
\sum_{k=1}^q\varphi_{jk} \bigg(\frac{1}{2}+{\rm i}r\bigg) \varphi_{kj}\bigg(\frac{1}{2}-{\rm i}r\bigg)=1.
\end{gather*}
Moreover,
\begin{gather*}
\Phi(1-s)=\Phi(s)^{-1}.
\end{gather*}
Since
\begin{gather*}
\frac{\rm d}{{\rm d}s}\log\det\Phi(s)=\operatorname{Tr}\big[\Phi'(s)\Phi^{-1}(s)\big],
\end{gather*}
and
\begin{gather*}
\frac{\rm d}{{\rm d}s}\log\det\Phi(s)= -\frac{\rm d}{{\rm d}s}\log\det\Phi(1-s)
\end{gather*}
we find that
\begin{gather*}
\operatorname{Tr}\bigg[\Phi'\bigg(\frac{1}{2}+{\rm i}r\bigg)
\Phi^{-1}\bigg(\frac{1}{2}+{\rm i}r\bigg) \bigg]
=\operatorname{Tr}\bigg[\Phi'\bigg(\frac{1}{2}-{\rm i}r\bigg)
\Phi^{-1}\bigg(\frac{1}{2}-{\rm i}r\bigg) \bigg].
\end{gather*}
This gives
\begin{align*}
\sum_{j=1}^q\iint_{F}\bigg|E_j^Y\bigg(z,\frac{1}{2}+{\rm i}r;n\bigg) \bigg|^2y^{2n}{\rm d}\mu(z)
={}&\frac{1}{2{\rm i}r}\operatorname{Tr}\bigg[\Phi\bigg(\frac{1}{2}-{\rm i}r\bigg)Y^{2{\rm i}r}-\Phi\bigg(\frac{1}{2}+{\rm i}r\bigg)Y^{-2{\rm i}r}\bigg]
\\
&+2q\log Y-\operatorname{Tr}\bigg[\Phi'\bigg(
\frac{1}{2}+{\rm i}r\bigg)\Phi^{-1}\bigg(\frac{1}{2}+{\rm i}r\bigg)\bigg].
\end{align*}
The assertion of the theorem follows.
\end{proof}

\section[Asymptotic behavior of the zeta function Z ell(s)]
{Asymptotic behavior of the zeta function $\boldsymbol{Z_{{\rm ell}}(s)}$}\label{Elliptic}
In this section, we want to compute the asymptotic behavior of $\log Z_{{\rm ell}}(s)$, where
\begin{gather*}
Z_{{\rm ell}}(s)=\prod_{j=1}^v \prod_{r=0}^{m_j-1}\Gamma\bigg(\frac{s+r}{m_j}\bigg)^{\frac{2\alpha_{m_j}(r-n)+1-m_j}{2m_j}}
\Gamma\bigg(\frac{s+2n+r}{m_j}\bigg)^{\frac{2\alpha_{m_j}(r+n)+1-m_j}{2m_j}},
\end{gather*}
when $u=s+n-\frac{1}{2}$ is large.
From the asymptotic behavior of $\log\Gamma(s)$ \eqref{gammabe}, we have
\begin{gather*}
\log Z_{{\rm ell}}(s)= \mathscr{A}u\log u+\mathscr{B}\log u+\mathscr{C}u+\mathscr{D}+o(1),
\end{gather*}
where
\begin{gather*}
\mathscr{A}= \sum_{j=1}^v\frac{\alpha_j}{m_j},\qquad
\mathscr{B}=\sum_{j=1}^v\beta_j,\qquad
\mathscr{C}= \sum_{j=1}^v\frac{1}{m_j} (-\alpha_j-\alpha_j \log m_j),
\\
\mathscr{D}= - \sum_{j=1}^v\beta_j\log m_j+\frac{1}{2}\log(2\pi)\sum_{j=1}^v \alpha_j,
\end{gather*}
and
\begin{gather*}
\alpha_j= \sum_{r=0}^{m_j-1} \bigg(\frac{2\alpha_{m_j}(r-n)+1-m_j}{2m_j} +\frac{2\alpha_{m_j}(r+n)+1-m_j}{2m_j}\bigg),
\\
\beta_j=\sum_{r=0}^{m_j-1}\bigg\{\frac{2\alpha_{m_j}(r-n)+1-m_j}{2m_j}
\bigg(\frac{-n+\frac{1}{2}+r}{m_j}-\frac{1}{2}\bigg)
\\ \hphantom{\beta_j=\sum_{r=0}^{m_j-1}\bigg\{}
{} +\frac{2\alpha_{m_j}(r+n)+1-m_j}{2m_j}\bigg(\frac{n+\frac{1}{2}+r}{m_j}-\frac{1}{2}\bigg)\bigg\}.
\end{gather*}
From \eqref{eq314_1}, we find that $\alpha_j=0$ for $1\leq j\leq 0$. Hence, $ \mathscr{A}=\mathscr{C}=0$, and
\begin{gather*}
\mathscr{D}= - \sum_{j=1}^v\beta_j\log m_j.
\end{gather*}
 It remains to calculate $\beta_j$. First we notice that
\begin{gather*}
\beta_j= \frac{1}{m_j}\sum_{r=0}^{m_j-1}\bigg\{r\frac{2\alpha_{m_j}(r-n)+1-m_j}{2m_j} +r\frac{2\alpha_{m_j}(r+n)+1-m_j}{2m_j} \bigg\}
\end{gather*}
 because of \eqref{eq314_1}. If $\alpha_{m_j}(n)=\ell$, then
 \begin{gather*}
 \alpha_{m_j}(r+n)=\begin{cases} \ell+r,&0\leq r\leq m_j-\ell-1,
 \\
 \ell+r-m_j,& m_j-\ell\leq r\leq m_j-1,\end{cases}
 \\
 \alpha_{m_j}(r-n)=\begin{cases} m_j+r-\ell,&0\leq r\leq\ell-1,
 \\ r-\ell,&\ell\leq r\leq m_j-1.\end{cases}
 \end{gather*}
With the help of a standard computer algebra, we find that
\begin{gather*}
\beta_j=\frac{m_j^2-1}{6m_j}-\frac{\alpha_{m_j}(n)(m_j-\alpha_{m_j}(n))}{m_j}.
\end{gather*}
Hence,
\begin{gather*}
\mathscr{B}=\sum_{j=1}^v\bigg(\frac{m_j^2-1}{6m_j}
-\frac{\alpha_{m_j}(n)(m_j-\alpha_{m_j}(n))}{m_j}\bigg),
\\
\mathscr{D}=-\sum_{j=1}^v\bigg(\frac{m_j^2-1}{6m_j}
-\frac{\alpha_{m_j}(n)(m_j-\alpha_{m_j}(n))}{m_j}\bigg)\log m_j.
\end{gather*}

\subsection*{Acknowledgements}
This research is supported by the Ministry of Higher Education Malaysia through the Fundamental Research Grant Scheme (FRGS) FRGS/1/2018/STG06/XMU/01/1. We would
like to thank L.~Takhtajan and J.~Friedman who have given helpful comments and
suggestions. We would also like to thank the referees for reading the paper carefully and giving valuable comments.

\pdfbookmark[1]{References}{ref}
\LastPageEnding

\end{document}